\documentclass[11pt]{amsart}
\usepackage{titletoc}
\usepackage{setspace}
\usepackage{mathtools}
\usepackage[margin=1in]{geometry}

\let\p\partial
\let\rh\hookrightarrow
\numberwithin{equation}{section}

\usepackage{amssymb,latexsym,amsfonts}
\usepackage{pb-diagram}
\usepackage{graphicx}
\usepackage{hyperref}
\usepackage{verbatim}
\usepackage{color} 
\usepackage[matrix,arrow]{xy}
\usepackage{marginnote}
\usepackage{tikz}

\usetikzlibrary{matrix}

\newcommand{\Ric}{\mathrm{Ric}}
\newcommand{\tr}{\mathrm{tr}}
\newcommand{\vol}{\mathrm{Vol}}
\newcommand{\spt}{\mathrm{spt}}

\theoremstyle{plain}

\newtheorem{lemma}{Lemma}
\newtheorem{theorem}{Theorem}
\newtheorem{cor}{Corollary}
\newtheorem{claim}{Claim}

\newtheorem*{m-lemma}{Main Lemma}

\newtheorem{step}{Step}
\theoremstyle{definition}
\newtheorem{definition}{Definition}
\newtheorem*{rem}{Remark}

\newcommand{\Z}{{\mathbf Z}}

\begin{document}

\vspace*{-5mm}

\title [Minimimal hypersurfaces and psc-bordism] {Minimimal
  hypersurfaces and bordism of \\ positive scalar
  curvature metrics}
\author{Boris Botvinnik}
\address{
Department of Mathematics\\
University of Oregon \\
Eugene, OR, 97405\\
USA
}
\email{botvinn@uoregon.edu}

\author{Demetre Kazaras}
\address{
Department of Mathematics\\
University of Oregon \\
Eugene, OR, 97405\\
USA
}
\email{demetre@uoregon.edu}


  \date{\today}
\begin{abstract}
Let $(Y,g)$ be a compact Riemannian manifold of positive scalar
curvature (psc).  It is well-known, due to Schoen-Yau, that any closed
stable minimal hypersurface of $Y$ also admits a psc-metric. We
establish an analogous result for stable minimal hypersurfaces with
free boundary. Furthermore, we combine this result with tools
from geometric measure theory and conformal geometry to study
psc-bordism.  For instance, assume $(Y_0,g_0)$ and
$(Y_1,g_1)$ are closed psc-manifolds equipped with stable minimal
hypersurfaces $X_0 \subset Y_0$ and $X_1\subset Y_1$.  Under natural
topological conditions, we show that a psc-bordism $(Z,\bar g) :
(Y_0,g_0)\rightsquigarrow (Y_1,g_1)$ gives rise to a psc-bordism
between $X_0$ and $X_1$ equipped with the psc-metrics given by
the Schoen-Yau construction.
\end{abstract}  
\maketitle
\renewcommand{\baselinestretch}{1.3}\normalsize
\vspace*{-10mm}

\begin{small}
\tableofcontents
\end{small}
\newpage

\section{Introduction}\label{sec1}
\subsection{Schoen-Yau minimal hypersurface technique}
For a Riemannian metric $g$ on a smooth manifold, we denote by $R_g$ the scalar curvature
function and by $H_{g}$ the mean curvature of the boundary (if it
is not empty). The Schoen-Yau minimal hypersurface technique
\cite{SY79} provides well-known geometric obstructions to the
existence of positive scalar curvature. Here is the first fundamental
result:
\begin{theorem}\label{thm1}
  {\rm \cite[Proof of Theorem 1]{SY79}}
Let $(Y,g)$ be a compact Riemannian manifold with $R_g>0$, and $\dim Y
=n \geq 3$. Let $X\subset Y$ be a smoothly embedded stable minimal hypersurface
with trivial normal bundle.  Then $X$ admits a metric $\tilde h$ with
$R_{\tilde h}>0$. Furthermore, the metric $\tilde h$ could be chosen to be conformal
to the restriction $g|_X$.
\end{theorem}
We note that Theorem \ref{thm1} is proven by analyzing the conformal
Laplacian of the hypersurface $X$. It it crucial that $X$ is
stable minimal.  For arbitrary $(Y,g)$ it is a
non-trivial (and possibly obstructed) problem to find a stable minimal hypersurface. 
However, in low dimensions, geometric measure theory can 
provide a source of stable minimal hypersurfaces.
\begin{theorem}\label{thm2}
  {\rm (See \cite[Chapter 8]{Morgan}, \cite[Theorem 5.4.15]{F})}
  Let $(Y,g)$ be a compact orientable Riemannian manifold with $3\leq
  \dim Y=n\leq 7$. Assume $\alpha \in H_{n-1}(Y;\Z)$ is a nontrivial
  element. Then there exists a smoothly embedded hypersurface $X\subset
  Y$ such that
\begin{enumerate}
\item[(i)] up to multiplicity, $X$ represents the class $\alpha$;  
\item[(ii)] $X$ minimizes volume among all hypersurfaces which represent
  $\alpha$ up to multiplicity. In particular, the hypersurface $X$ is stable minimal.
 \end{enumerate}
\end{theorem}
There are several important results based on Theorems \ref{thm1}
and \ref{thm2}. In particular, this gives a geometric proof that the
torus $T^n$ does not admit a metric of positive scalar curvature for
$n\leq 7$, see \cite{SY79}. This method was also crucial to 
provide first counterexample to the Gromov-Lawson-Rosenberg
conjecture, see \cite{Sch98}.  In this paper we extend these ideas
and techniques to the case of manifolds with boundary.

\subsection{Stable minimal hypersufaces with free boundary}
Let $(M,\bar g)$ be a manifold with non-empty boundary $\p M$ and
$W\subset M$ be an embedded hypersurface. We say that a hypersurface 
$W$ is \emph{properly embedded}
if, in addition, $\p W = \p M\cap W$.  Such a hypersurface $W\subset M$
is \emph{stable minimal with free boundary} if $W$ is a local minimum
of the volume functional among properly embedded
hypersurfaces, see Section \ref{sec:minintro}. We establish the following 
analogue of Theorem \ref{thm1} for
manifolds with boundary in Section \ref{sec:thm3}.
\begin{theorem}\label{thm3}
  Let $(M,\bar g)$ be a compact Riemannian manifold with non-empty
  boundary $\p M$, $R_{\bar g}>0$, $H_{\bar g}\equiv 0$,
  and $\dim M =n+1 \geq 3$. Let $W\subset M$ be an embedded 
  stable minimal
  hypersurface with free boundary and trivial normal bundle.
  Then $W$ admits a metric $\tilde h$ with $R_{\tilde h}>0$ and $H_{\tilde
    h}\equiv 0$. Furthermore, the metric $\tilde h$ could be chosen to
  be conformal to the restriction $\bar g|_W$.
\end{theorem}
The proof of Theorem \ref{thm3} is similar to the case of closed manifolds.
In particular, we have to analyze the
conformal Laplacian on $W$ with \emph{minimal boundary conditions}.
This boundary condition works well with the free boundary stability assumption.

For a compact oriented $(n+1)$-dimensional manifold $M$, we
consider the relative integral homology group $H_n(M,\p M;\Z)$. Let $\bar
\alpha\in H_{n}(M,\p M;\Z)$ be a non-trivial class which we may assume to be 
represented by a properly embedded hypersurface $W\subset M$.
We notice that the boundary $\p W$ (which may possibly be empty) represents
the class $\p(\bar \alpha)\in H_{n-1}(\p M;\Z)$, where $\p$ is the
connecting homomorphism in the exact sequence
\begin{equation}\label{eq2}
  \cdots\to H_{n}(\p M;\Z) \to H_{n}(M;\Z)
  \to H_{n}(M,\p M;\Z)\xrightarrow[]{\p }
H_{n-1}(\p M;\Z) \to \cdots
\end{equation}
There is an analog of Theorem \ref{thm2} which relies on a different regularity
result, see Appendix \ref{sec:thm4} for more details.
\begin{theorem}
{\rm{(See \cite[Theorem 5.2]{G})}}\label{thm4}
  Let $(M,\bar g)$ be a compact orientable Riemannian manifold with
  non-empty boundary $\p M$ and $3\leq \dim M=n+1\leq 7$.  Assume $\bar
\alpha \in H_{n}(M, \p M;\Z)$ is a nontrivial element. Then there
exists a smooth properly embedded hypersurface $W\subset M$ such that
\begin{enumerate}
\item[(i)] up to multiplicity, $W$ represents the class $\bar \alpha$;  
\item[(ii)] $W$ minimizes volume with respect to $\bar g$ 
among all hypersurfaces which represent $\bar \alpha$ up to multiplicity. 
  In particular, $W$ is stable minimal with free boundary.
\end{enumerate}
\end{theorem}

\subsection{Positive scalar curvature bordism and minimal hypersurfaces} 
The main result of this paper is an application of Theorems \ref{thm3} and 
\ref{thm4} to provide new obstructions for psc-metrics to be psc-bordant.
\begin{definition}\label{def01}
Let $(Y_0,g_0)$ and $(Y_1,g_1)$ be closed oriented $n$-dimensional
manifolds with psc-metrics. Then $(Y_0,g_0)$ and $(Y_1,g_1)$ are
{\emph{psc-bordant}} if there is a compact oriented
$(n+1)$-dimensional manifold $(Z,\bar g)$ such that
\begin{enumerate}
\item[$\bullet$] the manifold $Z$ is an oriented bordism between $Y_0$
  and $Y_1$, i.e., $\p Z=Y_0\sqcup -Y_1$;
\item[$\bullet$] $\bar g$ is a psc-metric which restricts to $g_i+dt^2$
  near the boundary $Y_i\subset \p Z$ for $i=0,1$.
\end{enumerate}
We write $(Z,\bar g): (Y_0,g_0) \rightsquigarrow (Y_1,g_1)$
for a psc-bordism as above.
\end{definition}
\begin{rem}
Sometimes we consider bordisms $(Z,\bar g): (Y_0,g_0)
\rightsquigarrow (Y_1,g_1)$ as above where the metrics do not
necessarily have positive scalar curvature.  However, we always
assume that the metric $\bar g$ restricts to a product metric near the
boundary.
\end{rem}  
Now we would like to enrich the psc-bordism relation with an extra
structure, namely with a choice of homology classes $\alpha_i\in
H_{n-1}(Y_i;\Z)$, $i=0,1$.  Recall the following elementary
observation.

Let $\alpha\in H_{n-1}(Y;\Z)$, where $Y$ is an oriented closed
$n$-dimensional manifold. Then the cohomology class $D\alpha\in H^1(Y;\Z)$
Poincare-dual to $\alpha$ can be represented by a smooth map $\gamma:
Y \to B\Z= S^1$. Furthermore, we can assume that a given point $s_0\in
S^1$ is a regular value for $\gamma$. It is easy to see that the
inverse image $X_{\gamma}:=\gamma^{-1}(s_0)\subset Y$ is an embedded
hypersurface which represents the homology class $\alpha$.

If $M$ is an oriented $(n+1)$-dimensional manifold
with a map $\bar \gamma : M \to
S^1$, let $\gamma : \p M \to S^1$ be the restriction $\bar \gamma|_{\p M}$. 
There is a simple relation between the
classes $[\bar \gamma]\in H^1(M;\Z)$ and $ [\gamma]\in H^1(\p M;\Z)$:
\begin{lemma}\label{lemma01}
Let $\bar \alpha\in H_n(M,\p M;\Z)$ and $\alpha\in H_{n-1}(\p M;\Z)$
be Poincare dual to the classes $[\bar \gamma]\in H^1(M;\Z)$ and
$[\gamma]\in H^1(\p M;\Z)$. Then $\p(\bar \alpha)= \alpha$, where $\p:
H_{n}(M,\p M;\Z)\to H_{n-1}(\p M;\Z)$ is the connecting homomorphism.
In particular, if $W=\bar\gamma^{-1}(s_0)\subset M$ is a smooth
properly embedded hypersurface representing $\bar\alpha$, then the
boundary $\p W$ represents the class $\alpha$.
\end{lemma}
\begin{definition}\label{bordism-alpha}
Let $(Y_0,g_0)$ and $(Y_1,g_1)$ be closed oriented $n$-dimensional
Riemannian manifolds with given maps $\gamma_0: Y_0\to S^1$ and
$\gamma_1: Y_1\to S^1$. We say that the triples $(Y_0,g_0,\gamma_0)$
and $(Y_1,g_1,\gamma_1)$ are \emph{bordant} if there exists a bordism
$(Z,\bar g): (Y_0,g_0) \rightsquigarrow (Y_1,g_1)$ and a map $\bar
\gamma : Z \to S^1$ such that $\bar \gamma|_{Y_i}= \gamma_i$ for
$i=0,1$.

If the metrics $g_0$, $g_1$ and $\bar g$ are psc-metrics, we say that
the triples $(Y_0,g_0,\gamma_0)$ and $(Y_1,g_1,\gamma_1)$ are
\emph{psc-bordant}.
In both cases we use the notation $(Z,\bar g,\bar \gamma):
(Y_0,g_0,\gamma_0) \rightsquigarrow (Y_1,g_1,\gamma_1)$ for such a
bordism.
\end{definition}
\begin{theorem}\label{thm5}
Let $(Y_0,g_0)$ and $(Y_1,g_1)$ be closed oriented connected
$n$-dimensional manifolds with psc-metrics, $3\leq n\leq 7$, and maps $\gamma_0
: Y_0\to S^1$ and $\gamma_1 : Y_1\to S^1$. Assume that
$(Y_0,g_0,\gamma_0)$ and $(Y_1,g_1,\gamma_1)$ are psc-bordant.

Then there exists a psc-bordism $(Z,\bar g,\bar \gamma):
(Y_0,g_0,\gamma_0) \rightsquigarrow (Y_1,g_1,\gamma_1)$ and a properly
embedded hypersurface $W\subset Z$ such that
\begin{enumerate}
\item[(i)] the hypersurface $W$ represents the class $\bar\alpha\in
  H_n(Z,\p Z;\Z)$ Poincare-dual to $[\bar \gamma]\in
  H^1(Z;\Z)$;
\item[(ii)] the hypersurface $X_i := \p W\cap Y_i\subset Y_i$
  represents the class $\alpha_i\in H_{n-1}(Y_i;\Z)$ 
  Poincare-dual to $[\gamma_i]\in H^1(Y_i;\Z)$, $i=0,1$;
\item[(iii)] there exists a metric $\bar h$ on $W$ such that $R_{\bar
  h}>0$ and $H_{\bar h}\equiv 0$ along $\p W$, and $R_{h_i}>0$, where
  $h_i=\bar h|_{X_i}$, in particular, $(W,\bar h): (X_0,h_0)
  \rightsquigarrow (X_1,h_1)$ is a psc-bordism;
\item[(iv)] the metric $\bar h$ on $W$ could be chosen to be conformal
  to the restriction $\bar g|_{W}$.
\end{enumerate}
\end{theorem}
\begin{rem}
The psc-bordism $(Z,\bar g,\bar \gamma)$ and
hypersurface $W$ may be chosen so that $\p W$ is arbitrarily $C^k$-close
to a desired homologically volume minimizing representative of 
$\alpha_0-\alpha_1$ for any $k$ and $i=0,1$.
\end{rem}
Recall few definitions. We say that a conformal class $C$ of metrics
is \emph{positive} if it contains a metric with positive scalar
curvature. It is equivalent to the condition that the Yamabe constant
$Y(X;C)>0$. Now let $W$ be a bordism with $\p W = X_0\sqcup X_1$, and
$C_0$, $C_1$ be positive conformal classes on $X_0$, $X_1$
respectively. Then we say that the conformal manifolds $(X_0,C_0)$ and
$(X_1,C_1)$ are positively conformally cobordant if the relative
Yamabe invariant $Y(W,X_0\sqcup X_1; C_0\sqcup C_1)>0$, see Section
\ref{sec4} for details. In these terms, the remark following Theorem \ref{thm5} 
can be used to show the following: 
\begin{cor}\label{corollary}
Let $(Y_0,g_0,\gamma_0)$ and $(Y_1,g_1,\gamma_1)$ be as in Theorem
\ref{thm5}. Assume $X_i\subset Y_i$ are volume minimizing hypersurfaces
representing homology classes Poincar\`e-dual to $[\gamma_i]\in H^1(X_i;\Z)$,
$i=0,1$. Then the conformal manifolds $(X_0,[g_0|_{X_0}])$ and
$(X_1,[g_1|_{X_1}])$ are positively conformally cobordant.
\end{cor}  
The first step in the proof of Theorem \ref{thm5} is to apply Theorem
\ref{thm4} to $\bar\alpha$, obtaining a minimal representative
$W$. The main difficulty is that $\p W$ is, in general, not a minimal
representative of $\p \bar\alpha$ and so we may not apply Theorem
\ref{thm1} to conclude that $\p W$ even admits a psc-metric.  However,
in Section \ref{sec3} we prove the Main Lemma, which states that $\p
W$ becomes closer to minimizing $\p\bar\alpha$ as longer collars are
attached to the psc-bordism $Z$.

This work was motivated by intense discussions with D. Ruberman and
N. Saveliev during and after the PIMS Symposium on Geometry and
Topology of Manifolds held in Summer 2015.  The authors are grateful
to D. Ruberman and N. Saveliev for their help and inspiration.  It is
a pleasure to also thank C. Breiner, A. Fraser, and T. Schick  for very helpful comments.

\section{Preliminaries and Theorem 3}
\subsection{Stable minimal hypersurfaces with free boundary}\label{sec:minintro}
Let $(M,\bar g)$ be a compact oriented $(n+1)$-dimensional 
Riemannian manifold with 
nonempty boundary $\p M$. Assume $W\subset M$ is a
properly embedded hypersurface. 

Let $\bar h$ denote the restriction metric $\bar h=\bar g|_{W}$
and fix a unit normal vector field $\nu^W$ on $W$ which is
compatible with the orientation. This determines the second
fundamental form $A^W$ on $W$ given by the formula
$A^W_{\bar g}(X,Y)=\bar
  g(\nabla_XY,\nu^W) $
for vector fields $X$ and $Y$ tangential to $W$.  The trace of $A^W_{\bar
  g}$ with respect to the metric $\bar h$ gives the mean curvature
$H^W_{\bar g}=\tr_{\bar h}A^W_{\bar g}$. We will often omit the sub-
and super-scripts, writing $\nu, A,$ and $H$ if there is no risk of
ambiguity.
\begin{definition}
Let $W\subset M$ be a properly embedded hypersurface. A \emph{variation} of
the hypersurface $W\subset M$ is a smooth one-parameter family
$\{F_t\}_{t\in(-\epsilon,\epsilon)}$ of proper embeddings $F_t:W\to
M$, $t\in(-\epsilon,\epsilon)$ such that $F_0$ coincides with the
inclusion $W\subset M$. A variation
$\{F_t\}_{t\in(-\epsilon,\epsilon)}$ is said to be \emph{normal} if
the curve $t\mapsto F_t(x)$ meets $W$ orthogonally for each $x\in W$.
\end{definition}

The vector field $X=\frac{d}{dt}F_t|_{t=0}$ is called the
{\emph{variational vector field}} associated to
$\{F_t\}_{t\in(-\epsilon,\epsilon)}$.  For normal
variations, the associated variational vector field takes the form
$\phi\cdot\nu^W$ for some function $\phi\in C^\infty(W)$.  Clearly, a
variation $\{F_t\}_{t\in(-\epsilon,\epsilon)}$ gives a smooth
function $t\mapsto \vol(F_t(W))$.
\begin{definition}
A properly embedded hypersurface $W\subset (M,\bar g)$ is \emph{minimal with
  free boundary} if
\begin{equation*}
\begin{array}{c}
  \left.\frac{d}{dt}\vol(F_t(W))\right|_{t=0} = 0
\end{array}
\end{equation*}
for all variations $\{ F_t\}_{t\in(-\epsilon,\epsilon)}$.
  \end{definition}
More notation: we denote by $d\sigma$ and $d\mu$ the volume forms of
$(W,\bar h)$ and $(\p W, h)$, where $h=\bar h|_{\p W}$ is the induced
metric. We denote the outward-pointing unit length
normal to $\p M$ by $\nu^\p$.  Below, Lemmas \ref{fact1} and \ref{lem:variation} 
contain well-known formulas, see \cite{FL}.
\begin{lemma}\label{fact1}
Let $(M,\bar g)$ be an oriented Riemannian manifold and let $W\subset
M$ be a properly embedded hypersurface. If
$\{F_t\}_{t\in(-\epsilon,\epsilon)}$ is a variation of $W$ with
variational vector field $X$, then
\begin{equation}\label{eq:1var}
\left.\frac{d}{dt}\vol(F_t(W))\right|_{t=0} 
=-\int_WH^W\bar g(X,\nu^W)d\mu+\int_{\p W}\bar g(X,\nu^{\p M})d\sigma .
\end{equation}
In particular, a hypersurface $W$ is minimal with free boundary if and
only if $H_{\bar g}^W\equiv 0$ and $W$ meets the boundary $\p M$
orthogonally.
\end{lemma}  
\begin{definition}
A properly embedded minimal hypersurface with free boundary $W$ is
\emph{stable} if
\begin{equation*}
\begin{array}{c}
\left.\frac{d^2}{dt^2}\vol(F_t(W))\right|_{t=0} \geq 0
\end{array}
\end{equation*}
for all variations $\{ F_t\}_{t\in(-\epsilon,\epsilon)}$. 
  \end{definition}
If a hypersurface $W$ is minimal with free boundary, then any
variational vector field must be parallel to $\nu^W$ on $\p W$ since
the variation must go through proper embeddings. Hence, it is enough to 
consider
only normal variations to analyze the second variation of the volume
functional. 
\begin{lemma}\label{lem:variation}
Let $(M,\bar g)$ be an oriented Riemannian manifold and let $W\subset
M$ be a properly embedded minimal hypersurface with free boundary.
Let $\{F_t\}_{t\in(-\epsilon,\epsilon)}$ be a normal variation with
variational vector field $\phi\cdot\nu^W$. Then
\begin{equation}\label{eq:2var}
\left.\frac{d^2}{dt^2}\vol(F_t(W))\right|_{t=0} \!
=\!\int_W\!\left(|\nabla\phi|^2-\phi^2(\Ric_{\bar g}
(\nu^W\!,\nu^W)+|A^W|^2)\right)d\mu
	-\int_{\p W}\!\phi^2A^{\p M}(\nu^W\!,\nu^W)d\sigma \ ,
\end{equation}
where $\Ric_{\bar g}$ denotes the Ricci tensor of $(M,\bar g)$.
\end{lemma}
It will be useful to rewrite equation (\ref{eq:2var}). The
Gauss-Codazzi equations for a minimal hypersurface $W\subset M$
imply
\begin{equation*}
R^M_{\bar g}=R^W_{\bar h}+2\Ric_{\bar g}(\nu^W,\nu^W)+|A^W|^2
\end{equation*}
on $W$. Here $R^M_{\bar g}$ and $R^W_{\bar h}$ are the scalar
curvatures of $(M,\bar g)$ and $(W,\bar h)$, respectively. It follows
that the inequality $\left.\frac{d^2}{dt^2}\vol(F_t(W))\right|_{t=0}
\geq 0$ is equivalent to
\begin{equation}\label{eq:gc2var}
  \int_W|\nabla\phi|^2d\mu\geq\int_W\frac12
  \phi^2\left(R^M_{\bar g}-R^W_{\bar h}+|A^W|^2\right)d\mu
	-\int_{\p W}\phi^2A^{\p M}(\nu^W,\nu^W)d\sigma.
\end{equation}
\subsection{Conformal Laplacian with minimal boundary conditions}
The proof of Theorem \ref{thm3} will rely on some basic facts about
the conformal Laplacian on manifolds with boundary. Let $(W,\bar h)$
be an $n$-dimensional manifold with non-empty boundary $(\p W, h)$ 
where $h=\bar h|_{\p W}$. We consider the following pair of operators acting on
$C^\infty(W)$:
\begin{equation*}\label{intr:eq3}
\left\{
\begin{array}{lcll}
  L_{\bar h} &=& -\Delta_{\bar h}+c_n R^W_{\bar h} 
 & \mbox{in $W$}
  \\
B_{\bar h} & = & \p_\nu+2c_nH_{\bar h}^{\p W}  & \mbox{on $\p W$},
\end{array}
\right.
\end{equation*}
where $\nu$ is the outward pointing normal vector to $\p W$ and
$c_n=\frac{n-2}{4(n-1)}$.

Recall that if $\phi\in C^\infty(W)$ is a positive
function, then the scalar and boundary mean curvatures of the
conformal metric $\tilde h=\phi^{\frac{4}{n-2}}\bar h$ are given by
\begin{equation}\label{eq:conformal}
\left\{
\begin{array}{lcll}
  R_{\tilde{h}} &=& c_n^{-1} \phi^{-\frac{n+2}{n-2}} \cdot L_{\bar
    h}\phi & \mbox{in $W$} \\ H_{\tilde{h}} & = &
  \frac{1}{2}c_n^{-1}\phi^{-\frac{n}{n-2}} \cdot B_{\bar h}\phi &
  \mbox{on $\p W$}.
\end{array}
\right.
\end{equation}
We consider a relevant Rayleigh quotient and take the infimum:
\begin{equation}\label{eq:Rayleigh}
\lambda_1=\inf_{\phi\not\equiv0\in H^1(W)}\frac{\int_W\left(|\nabla\phi|^2+c_nR^W_{\bar h}\phi^2\right)d\mu+2c_n\int_{\p W}H_{\bar h}^{\p W}\phi^2d\sigma}{\int_W\phi^2d\mu}.
\end{equation}
According to standard elliptic PDE theory, we obtain an elliptic
boundary problem, denoted by $(L_{\bar h}, B_{\bar h})$,
and the infimum $\lambda_1 = \lambda_1(L_{\bar h}, B_{\bar h})$ is
the \emph{principal eigenvalue of the minimal boundary problem}
$(L_{\bar h}, B_{\bar h})$. The corresponding Euler-Lagrange
equations are the following:
\begin{equation}
\label{eq:eigenvalue}
\left\{
\begin{array}{lcll}
  L_{\bar h}\phi &= & \lambda_1\phi& \text{ in }W
  \\
  B_{\bar h}\phi &= & 0 &\text{ on } \p W.
\end{array}
\right.
\end{equation}  
This problem was first studied by Escobar \cite{E1} in the context of
the Yamabe problem on manifolds with boundary.  

Let $\phi$ be a solution of (\ref{eq:eigenvalue}).  It is
well-known that the eigenfunction $\phi$ is smooth and can be chosen to
be positive.  
A straight-forward computation shows that the conformal metric $\tilde h=
\phi^{\frac{4}{n-2}}\bar h$ has the following scalar and mean
curvatures:
\begin{equation}\label{eq:6}
  \left\{
\begin{array}{lcll}
  R_{\tilde  h} &= & \lambda_1 \phi^{-\frac{4}{n-2}}_1 & \mbox{in $W$}
\\
  H_{\tilde  h} &\equiv &  0 & \mbox{on $\p W $}.
\end{array}
\right. 
\end{equation} 
In particular, the sign of the eigenvalue $\lambda_1$ is a 
conformal invariant, see
\cite{E1,E3}.

\subsection{Proof of Theorem \ref{thm3}}\label{sec:thm3}
  Let $(M,\bar g)$ and
  $W\subset M$ be as in Theorem \ref{thm3}.  From the assumption
  $H^{\p M}\equiv0$, one can use the Gauss equations to show that
  $A^{\p M}(\nu,\nu)=-H^{\p W}$ where $H^{\p W}$ is the mean curvature
  of $\p W$ as a hypersurface of $W$.  Now, using the condition
  $R^M_{\bar g}>0$, the stability inequality (\ref{eq:gc2var}) implies
\begin{equation}\label{eq:thm3-1}
  \int_W\left(|\nabla\phi|^2+\frac12 R^W_{\bar h}\right)d\mu+\int_{\p W}\phi^2H^{\p W}d\sigma\geq0
\end{equation}
for all functions $\phi\in H^1(W)$ with strict inequality if
$\phi\not\equiv0$.  By simple manipulation, the inequality
(\ref{eq:thm3-1}) may be written as
\begin{equation}\label{eq:thm3var}
\int_W\left(|\nabla\phi|^2+c_nR^W_{\bar h}\right)d\mu+2c_n\int_{\p
  W}\phi^2 H^{\p W}d\sigma > (1-2c_n)\int_W|\nabla\phi|^2d\mu
\end{equation}
for all $\phi\not\equiv0\in H^1(W)$.  The right hand side of
(\ref{eq:thm3var}) is non-negative since $1-2c_n=\frac{n}{2(n-1)}>0$.
Furthermore, the left hand side of (\ref{eq:thm3var}) coincides with
the numerator of the Rayleigh quotient in equation
(\ref{eq:Rayleigh}).  We conclude that the principal eigenvalue
$\lambda_1= \lambda_1(L_{\bar h},B_{\bar h})$ is positive. Let $\phi$
be an eigenfunction corresponding to $\lambda_1$.  Then, according to
(\ref{eq:6}), the metric $\tilde h= \phi^{\frac{4}{n-2}}\bar h$ has
positive scalar curvature and zero mean curvature on the boundary.
This completes the proof of Theorem \ref{thm3}.

\section{Cheeger-Gromov convergence of minimizing hypersurfaces}\label{sec3}
\subsection{Convergence of hypersurfaces}
Here we introduce the notion of smooth convergence of hypersurfaces we
require for the proof of Theorem \ref{thm5}.  First, we consider the
case when the hypersurfaces are embedded in the same ambient
$(n+1)$-dimensional manifold $M$.  Below we use coordinate charts
$\Phi_j : U_j\to M$, where $U_j$ is an open subset of
$\mathbb{R}^{n+1}_+=\{(x_1,\ldots,x_{n+1})\in\mathbb{R}^{n+1}\colon
x_{n+1}\geq 0\}$.

\begin{figure}[!htb]
\begin{picture}(0,0)
\put(94,53){{\small $P$}}
\put(74,43){{\small $x$}}
\put(5,65){{\small $\eta$}}
\put(40,73){{\small $x\!+\!u(x)\eta $}}
\put(88,10){{\small $U$}}
\end{picture}
\includegraphics[height=1.7in]{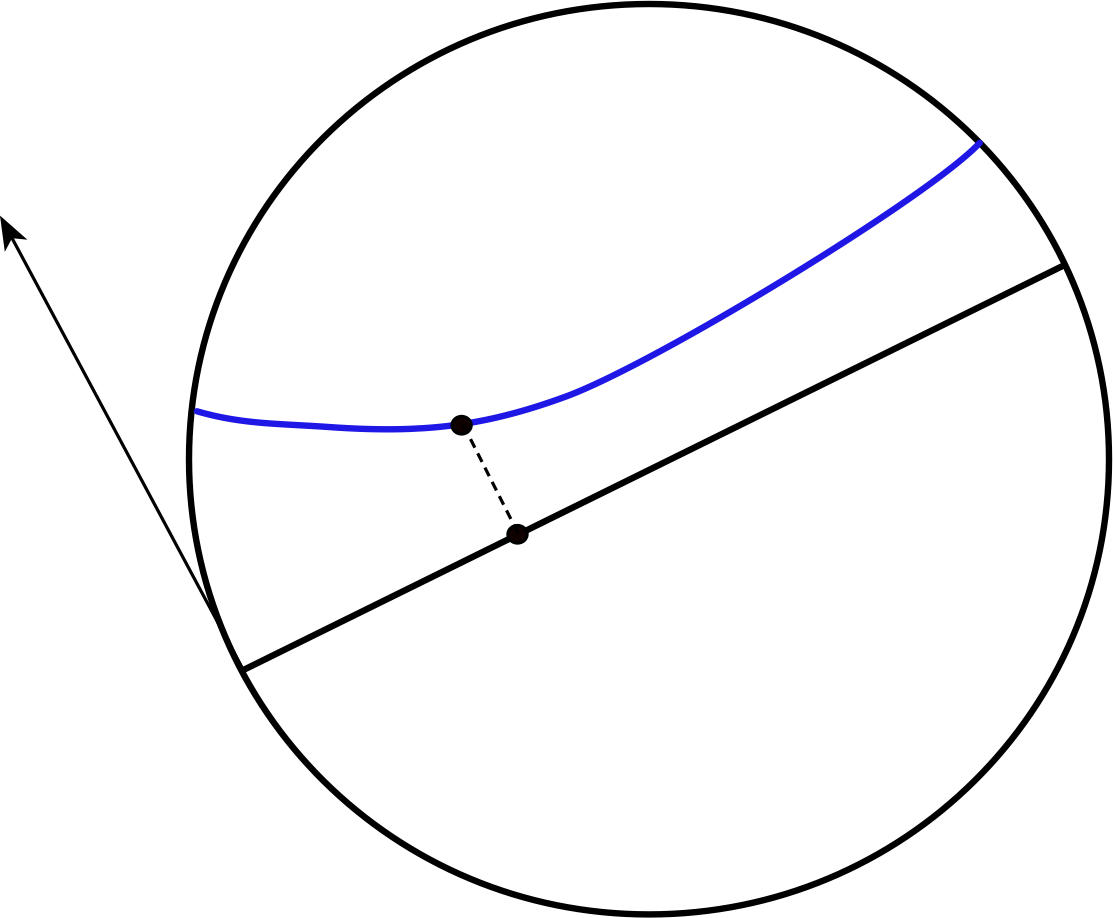}
\caption{$\mathrm{graph}(u)$}
\label{fig:01}
\end{figure} 

Let $P\subset \mathbb{R}^{n+1}$ be a hyperplane equipped with a normal
unit vector $\eta$, and $U\subset \mathbb{R}^{n+1}_+$ be an open
subset. Then for a function $u : P\cap U \to \mathbb{R}$, we denote by
$\mathrm{graph}(u)$ its \emph{graph}, see Fig. \ref{fig:01}:
\begin{equation*}
\mathrm{graph}(u)=\{x+u(x)\eta \ | \ x\in P\cap U \ \}.
\end{equation*}  
\begin{definition}\label{def:local-graph}
Let $k \geq 1$ be an integer.  Let $(M,\bar g)$ be an
$(n+1)$-dimensional compact Riemannian manifold and let
$\{\Sigma_i\}_{i=1}^\infty$ be a sequence of smooth, properly embedded
hypersurfaces. Then we say that the sequence
$\{\Sigma_i\}_{i=1}^\infty$ converges to a smooth embedded
hypersurface $\Sigma_\infty$ \emph{$C^k$-locally as graphs} if there
exist
\begin{enumerate}
\item[(i)] coordinate charts $\Phi_j:U_j\to M$ for $j=1,\dots,N$;
\item[(ii)] hyperplanes $P_j\subset \mathbb{R}^{n+1}$ 
  equipped with unit normal vectors $\eta_j$ for $j=1,\dots,N$;
\item[(iii)] smooth functions $u_{i,j}: P_j\cap U_j \to\mathbb{R}$ for
  $j=1,\dots,N$, $i=1,2,\ldots$, and $i=\infty$,
\end{enumerate}
which satisfy the following conditions:
\begin{enumerate}
\item[(a)] $\displaystyle
  \bigcup_{j=1}^N\Phi_j(\mathrm{graph}(u_{i,j})\cap U_j)=\Sigma_i$ for
  $i=1,2,\ldots$ and $i=\infty$;
\item[(b)] for each $j=1,\ldots, N$, $u_{i,j}\to u_{\infty,j}$ 
  in the $C^k(P_j\cap U_j)$ topology as
  $i\to\infty$.
\end{enumerate}
We say the sequence $\{\Sigma_i\}_{i=1}^\infty$ converges to a smooth
embedded hypersurface $\Sigma_\infty$ \emph{smoothly locally as
  graphs} if it converges $C^k$-locally as graphs for all
$k=1,2,\ldots$.
\end{definition}
Next, we consider a sequence $\{(M_i, \Sigma_i, \bar g_i,
\mathsf{S}_i)\}_{i=1}^\infty$, where $(M_i, \bar g_i)$ is a Riemannian manifold,
$\Sigma_i\subset M_i$ is a properly embedded smooth hypersurface, and
$\mathsf{S}_i\subset M_i$ a compact subset, playing a role of a
base-point or a finite collection of base points. 
\begin{definition}\label{def:graph}
Let $k \geq 1$ be an integer, and $\{(M_i, \Sigma_i, \bar g_i,
\mathsf{S}_i)\}_{i=1}^\infty$ be a sequence as above, where $\dim
M_i=n+1$. We say that $\{(M_i, \Sigma_i, \bar g_i,
\mathsf{S}_i)\}_{i=1}^\infty$ $C^k$-converges to
$(M_\infty,\Sigma_\infty,\bar g_\infty, \mathsf{S}_\infty)$ if there
is an exhaustion of $M_\infty$ by precompact open sets
\begin{equation*}
\mathsf{S}_\infty\subset \mathsf{U}_1\subset
\mathsf{U}_2\subset\dots\subset M_\infty, \quad
M_\infty=\bigcup_{i=1}^\infty \mathsf{U}_i
\end{equation*} 
and maps $\Psi_i:\mathsf{U}_i\to M_i$ which are diffeomorphisms onto
their images for each $i=1,2,\ldots$, such that
\begin{enumerate}
\item[(1)]
  $\mathrm{dist}^{M_\infty}_H(\mathsf{S}_\infty,\Psi_i^{-1}(\mathsf{S}_i))\to
  0$ as $i\to\infty$, where $\mathrm{dist}^{M_\infty}_H$ is the Hausdorff distance
  for subsets of the manifold $M_\infty$;
\item[(2)] the sequence $\{\Psi_i^*\bar g_i\}$ of metrics converges
  to $\bar g_\infty$
  in the ${C^{k}(\mathsf{U}_i)}$-topology as $i\to\infty$;
\item[(3)] the sequence of hypersurfaces
  $\{\Psi_j^{-1}(\Sigma_i)\}_{i=1}^\infty$ converges $C^k$-locally as
  graphs in the manifold $M_\infty$ to $\Sigma_\infty\cap
  \mathsf{U}_j$ as $i\to\infty$ for each $j=1,\dots,N$.
\end{enumerate}
\end{definition}
\begin{rem}
 We notice that the conditions $(1)$ and $(2)$ imply that the sequence
$\{(M_i,\bar g_i,\mathsf{S}_i)\}_{i=1}^\infty$ $C^k$-converges to $(M_\infty,\bar
g_\infty,\mathsf{S}_\infty)$ in the Cheeger-Gromov topology. 
\end{rem}
We say that $\{(M_i,\Sigma_i,\bar g_i,\mathsf{S}_i)\}_{i=1}^\infty$
\emph{smoothly converges} to $(M_\infty,\Sigma_\infty,\bar
g_\infty,\mathsf{S}_\infty)$ if it $C^k$-converges for all $k\geq 1$.
Then we say that $\{(M_i, \Sigma_i, \bar g_i,
\mathsf{S}_i)\}_{i=1}^\infty$ \emph{sub-converges} to
$(M_\infty,\Sigma_\infty,\bar g_\infty, \mathsf{S}_\infty)$ if it has
a subsequence which converges to $(M_\infty,\Sigma_\infty,\bar
g_\infty,\mathsf{S}_\infty)$.  In this case we write
\begin{equation*} (M_i,
\Sigma_i, \bar g_i, \mathsf{S}_i) \longrightarrow
(M_\infty,\Sigma_\infty,\bar g_\infty, \mathsf{S}_\infty).
\end{equation*} 

\subsection{Main convergence result}\label{sec3-2}
We are ready to set the stage for the main result of this section.
Let $(Y,g)$ be a closed, oriented $n$-dimensional Riemannian manifold with a
homology class $\alpha\in H_{n-1}(Y;\mathbb{Z})$. As we discussed in
Section \ref{sec1}, the class $\alpha$ gives the Poincar\`e dual class
$D\alpha= [\gamma]\in H^1(Y;\mathbb{Z})$ represented by some map
$\gamma : Y \to S^1$. Furthermore, we assume that there is a bordism
\begin{equation}\label{eq:bordism}
(M,\bar g, \bar \gamma) : (Y,g,\gamma ) \rightsquigarrow
  (Y',g',\gamma')
\end{equation}  
for some triple $(Y',g',\gamma')$. In the above, $\bar \gamma : M\to
S^1$ represents a class $[\bar \gamma]\in H^1(M;\mathbb{Z})$
Poincar\`e dual to a class $\bar\alpha\in H_{n}(M,\p M;\mathbb{Z})$.

Recall that $Y\subset\p M$ and $\bar g=g+dt^2$ near $Y$.  For a 
real number $L\geq 0$, we consider the following Riemannian manifold
\begin{equation*}
(M_L, \bar g_L) : = (M\cup_{Y\times\{-L\}}(Y\times[-L,0]), \bar g_L),
\end{equation*}  
where $\bar g_L$ restricts to $\bar g$ on $M$ and to the
product-metric $g+dt^2$ on $Y\times[-L,0]$.  We obtain another bordism
\begin{equation}\label{eq:bordism-L}
(M_L,\bar g_L, \bar \gamma_L) : (Y,g,\gamma ) \rightsquigarrow
  (Y',g',\gamma'),
\end{equation}  
where $[\bar \gamma_L]$ is the image of $[\bar\gamma]$ under the isomorphism
$H^1(M;\mathbb{Z})\cong H^1(M_L;\mathbb{Z})$. We refer to the bordism 
$(M_L,\bar g_L, \bar \gamma_L)$ as
the \emph{$L$-collaring of} $(M,\bar g, \bar \gamma)$.
Below we will take $L$ be an integer $i= 1,2,\ldots$,
and write $\bar \alpha_L\in H_{n}(M,\p M; \mathbb{Z})$ for the
class Poincar\`e dual to $[\bar \gamma_L]$.
\begin{m-lemma}\label{lem:main-1}
Let $(M,\bar g, \bar \gamma) : (Y,g,\gamma ) \rightsquigarrow
(Y',g',\gamma')$ be a bordism as in {\rm (\ref{eq:bordism})} and denote by
$(M_i,\bar g_i,\bar \gamma_i)$ the $i$-collaring of $(M,\bar g, \bar
\gamma)$ as in {\rm (\ref{eq:bordism-L})} for $i=0,1,2,\ldots$. Fix a
basepoint in each component of $Y$, denote their union by
$\mathsf{S}$, and let $\mathsf{S}_i$ be the image of $\mathsf{S}$ under
the inclusion
\begin{equation*}
  Y\cong Y\times \{0\} \subset Y\times [-i,0] \subset  M_i.
\end{equation*}
Assume
$W_i\subset M_i$ is an oriented homologically volume minimizing
representative of $\bar \alpha_i$ for $i=0, 1,2,\ldots$.
If $X\subset Y$ is an embedded hypersurface which is the only
volume minimizing representative of $\alpha\in H_{n-1}(Y;\mathbb{Z})$, then
there is smooth subconvergence
\begin{equation*}
(M_i,W_i,\bar g_i,\mathsf{S}_i) \longrightarrow
	(Y\times(-\infty,0], X\times(-\infty,0],g+dt^2,\mathsf{S}_\infty)
\end{equation*}
as $i\to\infty$ where $\mathsf{S}_\infty\subset Y\times\{0\}$
is the inclusion of $\mathsf{S}$.
\end{m-lemma}
\begin{rem}
In Main Lemma, we allow the manifold $Y'$ to be empty. 
\end{rem}  
\subsection{Proof of the Main Lemma: outline}
Consider the limiting space $Y\times(-\infty,0]$, with the
exhaustive sequence $\mathsf{U}_i=Y\times(-i-1,0]$ and maps
$\Psi_i:\mathsf{U}_i \to M_i$ taking $\mathsf{U}_i$ identically
onto $Y\times(-i-1,0]\subset M_i$.  Our choice of $\mathsf{U}_i$ and $\Psi_i$ satisfy the conditions
$(1)$ and $(2)$ from Definition \ref{def:graph} for obvious reasons.

It will be useful to equip $M$ with a height function $F: M\to[-1,0]$
satisfying $Y=F^{-1}(0)$ and $Y'= F^{-1}(-1)$.
Extend this function to $M_i$ by 
\begin{equation*}
  F_i(x)= 
  \begin{cases}
\ \ t &\text{ if }x=(y,t)\in Y\times[-i,0]\\
F(x)-i&\text{ if }x\in M.
\end{cases}
\end{equation*}
\begin{figure}[!htb]
\begin{picture}(0,0)
\put(-20,190){{\small $Y$}}
\put(73,158){{\small \textcolor{red}{$W_i^R$}}}
\put(28,0){{\small $M_i$}}
\put(180,190){{\small $0$}}
\put(180,127){{\small $-R$}}
\put(180,96){{\small $-R-1$}}
\put(180,-2){{\small $-i-1$}}
\put(150,114){{\small $F_i$}}
\end{picture}
\includegraphics[height=3in]{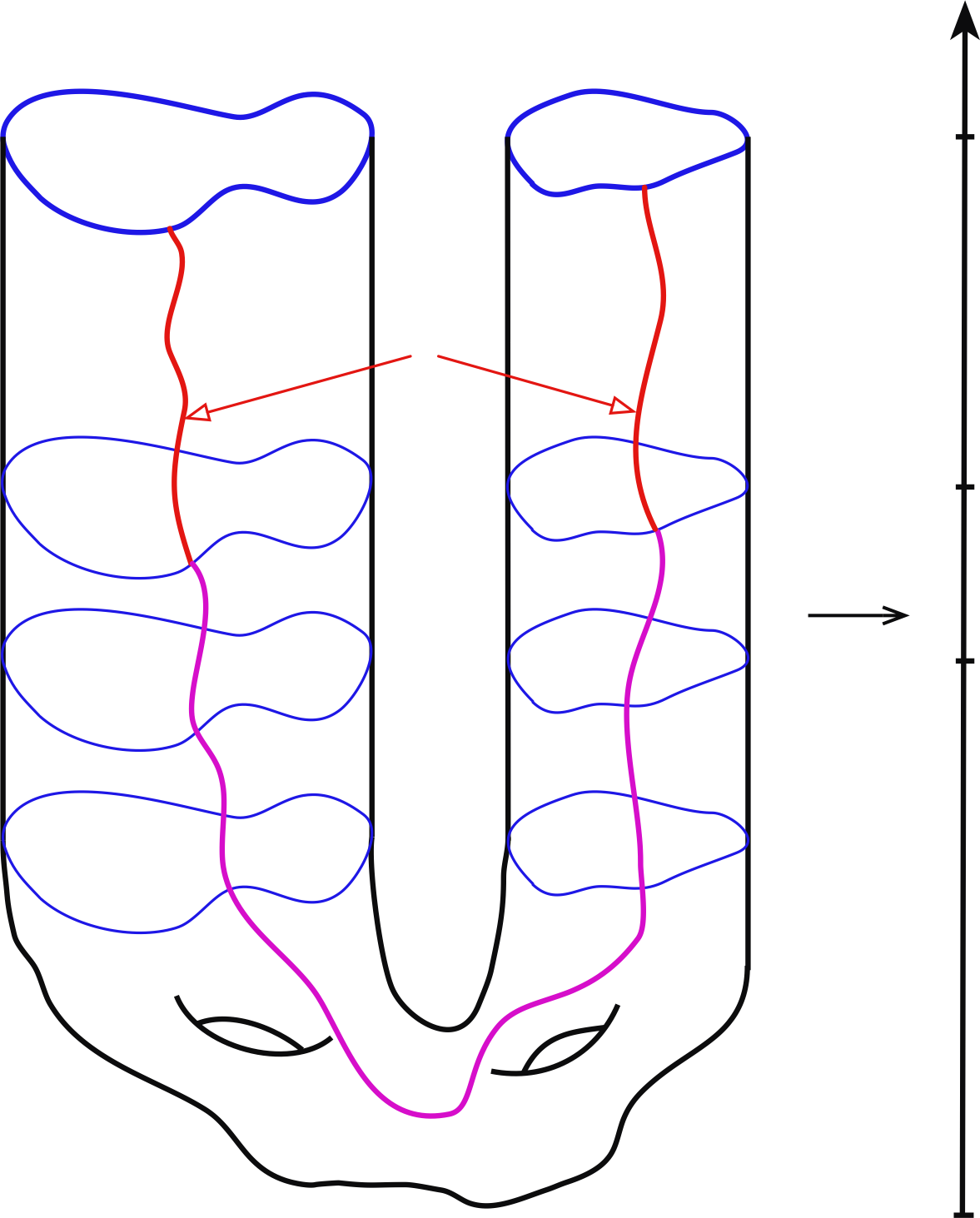}
\caption{The hypersurface $W_i^R\rh M_i$. In this figure, $Y'=\emptyset$.}
\label{fig:02}
\end{figure} 

\noindent
For any positive integer $i$ and heights $0\leq R<R'\leq i$, we
write
\begin{equation*}
W_i^R=F_i^{-1}([-R,0]) \ \ \ \mbox{and}
\ \ \ W_i[-R',-R]=F_i^{-1}([-R',-R]).
\end{equation*}
Let $\alpha\in H_{n-1}(Y;\mathbf{Z})$ be the class from the statement
of Main Lemma. For $L>0$ let 
\begin{equation*}
  \alpha\times [-L,0] \in H_n(Y\times[-L,0],Y\times\{-L,0\};\mathbf{Z})
\end{equation*}
be the product of $\alpha$ and
the fundamental class of $([-L,0],\{-L,0\})$.
We will break up the proof of Main Lemma into three claims.
\begin{claim}\label{claim1}
Let $L>0$. The hypersurface $X\times[-L,0]\subset Y\times[-L,0]$
is the only homologically volume-minimizing representative of 
$\alpha\times [-L,0]$.
\end{claim}
\begin{claim}\label{claim2}
For each $R>0$, $\vol(W_i^R)\to R\cdot\vol(X)$ as $i\to\infty$.
\end{claim}
\begin{claim}\label{claim3}
For each $R>0$, there is a sequence $\{a^R_i\}_{i=1}^\infty$ 
such that, for each $j=1,2,\ldots$, the hypersurfaces
$\{\Psi_j^{-1}(W_{a^R_i}^R)\}_{i=1}^\infty$ converge
smoothly locally as graphs in $Y\times(-\infty,0]$.
\end{claim}
Now we show how Main Lemma follows from 
Claims \ref{claim1}, \ref{claim2}, and \ref{claim3}.
Indeed, by Claim \ref{claim3}, for each $k=1,2,\ldots$, there is a
sequence $\{a^k_i\}_{i=1}^\infty$ such that, for each $j=1,2,\ldots$,
the hypersurfaces $\{\Psi_j^{-1}(W_{a^k_i}^k)\}_{i=1}^\infty$
converges smoothly locally as graphs to some hypersurface
\begin{equation*}
W_{\infty,k}\subset
Y\times(-\infty,0].
\end{equation*}
We notice that the hypersurface $W_{\infty,k}$ is contained in
$Y\times[-k,0]$ and represents the class $\alpha\times[-k,0]$.
Since the convergence is smooth, we have
\begin{align}
  \vol(\Psi_j^{-1}(W_{\infty,k}))&=\lim_{i\to\infty}\vol(\Psi_j^{-1}(W_{a^k_i}^k)) =
k\cdot\vol(X)\notag,
\end{align}
where the last equality follows from Claim \ref{claim2}.  However, according to
Claim \ref{claim1}, the only volume minimizing representative of
$\alpha\times [-k,0]$ is the hypersurface $X\times[-k,0]$ which has
the volume $k\cdot\vol(X)$.
Thus $W_{\infty,k}$ must be $X\times[-k,0]$.  Evidently, the diagonal
sequence $\{\Phi_j^{-1}(W_{a_i^i})\}_{i=1}^\infty$ has the property
that, for each $k>0$, $\Phi_j^{-1}(W_{a_i^i}^k)$ converges smoothly
locally as graphs to $X\times[-k,0]$. This then completes the proof of
Main Lemma.
\subsection{Proof of Claim \ref{claim1}}
Let $\Sigma\subset Y\times[-L,0]$ be a properly embedded hypersurface 
representing the class
$\alpha\times[-L,0]$. 
Consider the projection function $P:\Sigma\to[-L,0]$.
The coarea formula \cite[Theorem 5.3.9]{KP} applied to $P$ yields
\begin{equation}\label{eq:coarea}
\int_{\Sigma}|\nabla P|d\mu=
\int_{-L}^0\mathcal{H}^{n-1}(P^{-1}(t))dt \ ,
\end{equation}
where $\mathcal{H}^{n-1}$ denotes the $(n-1)$-dimensional Hausdorff measure 
associated to the metric
$h+dt^2$ on $Y\times[-L,0]$. Notice that $P$ is weakly contractive
in the sense that
\begin{equation*}
  |P(x)-P(y)|\leq\mathrm{dist}^{\Sigma}(x,y)
\end{equation*}
for all $x,y\in\Sigma$. Thus we have the pointwise bound $|\nabla
P|\leq1$. Furthermore, since $P^{-1}(t)$ represents the class
$\alpha\in H_{n-1}(Y\times\{t\};\Z)$ for each $t\in[-L,0]$, 
\begin{equation*}
  \mathcal{H}^{n-1}(P^{-1}(t))\geq\vol(X)
\end{equation*} with equality if
and only if $P^{-1}(t)$ is $X$.  Combining this observation with
(\ref{eq:coarea}), we conclude
\begin{equation*}
\vol(\Sigma)\geq L\cdot\vol(X)
\end{equation*}
with equality if and only if $\Sigma=X\times[-L,0]$. 
This completes the proof of Claim \ref{claim1}.
\subsection{Proof of Claim \ref{claim2}}
Before we begin, we will construct particular (in general, non-minimizing)
properly embedded 
hypersurfaces $N_L\subset M_L$ representing $\alpha_L$ with which to
compare $\vol(W_L)$ against.

Let $X\subset Y$ and $W_0\subset M_0$ be as in Main Lemma. 
Since $\p W_0 \cap Y$ and $X$ represent
the same homology class, they are bordant via a smooth, properly
embedded hypersurface $\iota: U \rh Y\times[0,1]$. We identify
$[0,1]\cong[-L,-L+ 1]$ to obtain the embedding
\begin{equation*}
\iota_L : U \xhookrightarrow[]{\iota} Y\times[0,1] \cong
Y\times[-L,-L+ 1] \rh M_L.
\end{equation*}
Clearly the embedding $\iota: U \rh Y\times[0,1]$ may be chosen so that
\begin{equation*}
 N_L:= W_0 \cup_{\p W_0} U_L \cup (X \times [-L+1,0]),
\end{equation*}
where $U_L= \iota_L(U)$, is a smooth properly embedded hypersurface
of $M_L$.

Evidently, $\vol(N_L)=\vol(W_0)+\vol(U_L)+(L-1)\vol(X)$ and $N_L$
represents the same homology class as $W_L$.  Since $W_L$ is
homologically area-minimizing, we have $\vol(W_L)\leq\vol(N_L)$.
\begin{figure}[!htb]
\begin{picture}(0,0)
\put(-20,195){{\small $Y$}}
\put(69,166){{\small \textcolor{red}{$W_L^R$}}}
\put(69,124){{\small \textcolor{blue}{$N_L$}}}
\put(70,90){{\small \textcolor{red}{$U_L$}}}
\put(58,-10){{\small $M_L$}}
\put(360,193){{\small $0$}}
\put(360,135){{\small $-R$}}
\put(360,106){{\small $-L+1$}}
\put(360,75){{\small $-L$}}
\put(230,-10){{\small $Y\times (-\infty,0]$}}
\end{picture}
\includegraphics[height=3in]{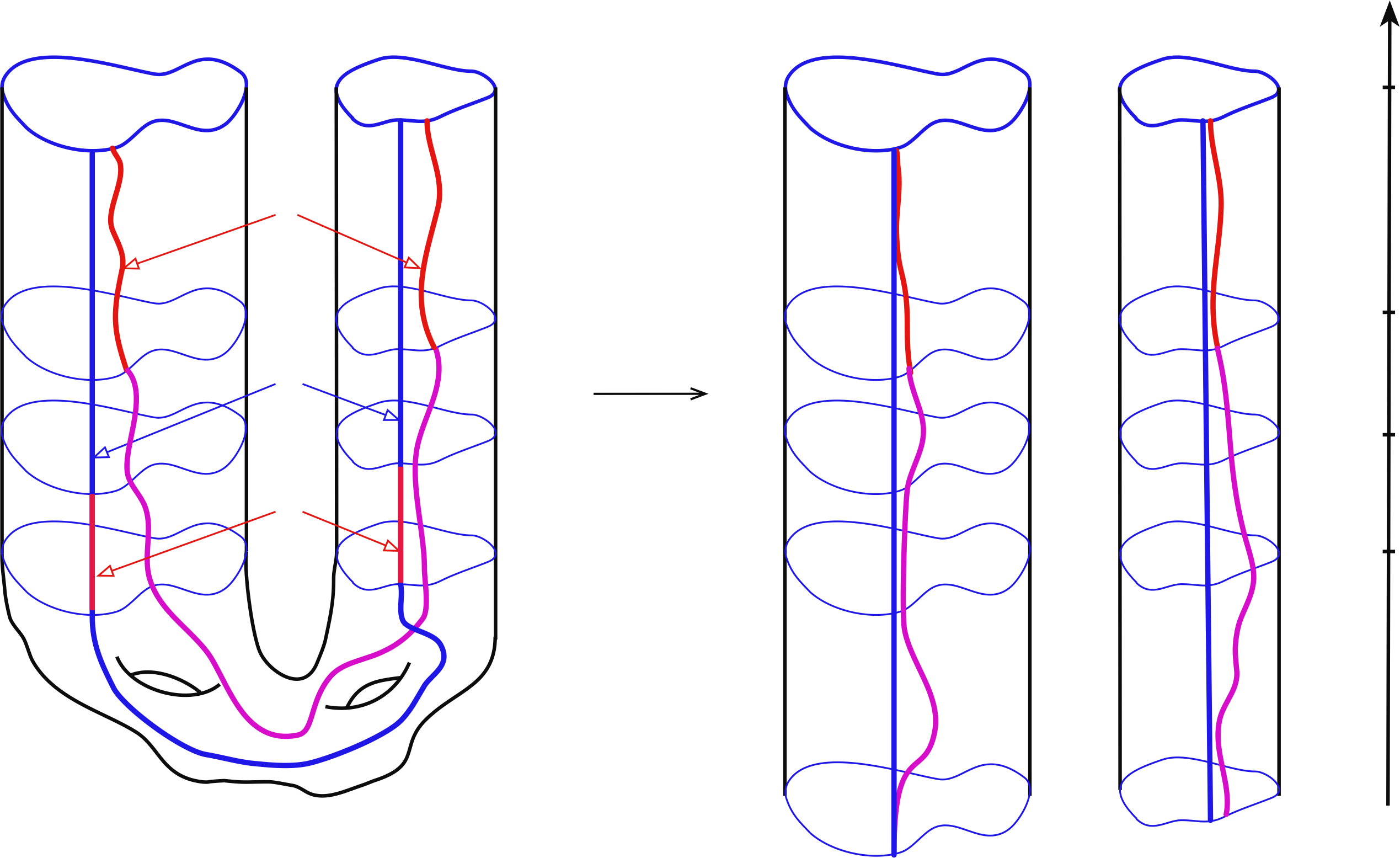}
\vspace*{2mm}

\caption{The hypersurface $N_L\rh M_L$.}
\label{fig:03}
\end{figure} 

\noindent
In other words, we obtain the inequality
\begin{equation}\label{eq:comparison}
\vol(W_L^{R})+\vol(W_L\setminus W_L^{R})\leq\vol(W_0)+\vol(U_L)+(L-1)\vol(X)
\end{equation}
for any $0<R<L-1$.

Now we are ready to prove Claim \ref{claim2}. Assume it
fails. Then there exist $\epsilon_0,R_0>0$ and an increasing sequence
of whole numbers
$\{a_i\}_{i=1}^\infty$ such that the inequality
\begin{equation}\label{eq:claim2vol}
\vol(W_{a_i}^{R_0})> R_0\cdot\vol(X)+\epsilon_0
\end{equation}
holds for all $i$. Combining the inequality (\ref{eq:comparison}) with
the assumption (\ref{eq:claim2vol}), we have
\begin{equation}\label{eq:inequality1}
\vol(W_0)+\vol(U_{a_i})+(a_i-1)\vol(X)>\vol(W_{a_i}\setminus W_{a_i}^{R_0})+\epsilon_0+R_0\vol(X).
\end{equation}
Now we will inspect the first term in the right hand side of
(\ref{eq:inequality1}):
\begin{align}\label{eq:inequality2}
\vol(W_{a_i}\setminus W_{a_i}^{R_0})&=
	\vol(W_{a_i}[a_{i-1}-a_i,-R_0])+\vol(W_{a_i}[-a_i-1,a_{i-1}-a_i])\notag\\
&\geq(a_i-a_{i-1}-R_0)\vol(X)+\vol(W_{a_{i-1}})\notag\\
&>(a_i-a_{i-1})\vol(X)+\epsilon_0+\vol(W_{a_{i-1}}\setminus W_{a_{i-1}}^{R_0}).
\end{align}
Here we use Claim \ref{claim1} in the first inequality and the
assumption (\ref{eq:claim2vol}) in the second.

Combining (\ref{eq:inequality1}) with (\ref{eq:inequality2}), we obtain 
\begin{equation*}
\vol(W_0)+\vol(U_{a_i})+(a_i-1)\vol(X)>(a_i-a_{i-1}+R_0)\vol(X)+2\epsilon_0+
	\vol(W_{a_{i-1}} \setminus W_{a_{i-1}}^{R_0}).
\end{equation*}
We iterate the argument to find 
\begin{equation}\label{eq:inequality3}
\vol(W_0)+\vol(U_{a_i})+(a_1-R_0-1)\vol(X)>i \cdot\epsilon_0+
	\vol(W_{a_1})
\end{equation}
for every $i=1,2,\ldots$. Since the left hand side of (\ref{eq:inequality3}) is
independent of $i$, we arrive at a contradiction
by taking $i$ to be sufficiently large.
\subsection{Proof of Claim \ref{claim3}}\label{section:claim3}
While the proof of Claim \ref{claim3} is rather technical, it is
essentially a consequence of standard tools used in the study of
stable minimal hypersurfaces. For instance, see \cite{CS}
for a similar result in a $3$-dimensional context.
We divide the proof into three steps,
referring to Appendix A when necessary.

To begin, we require the following straight-forward volume bound.
\begin{step}\label{step:volume}
For each $R>0$, there is a constant $V_R>0$ such that 
\begin{equation*}
\vol(W_i[-\lambda-R,-\lambda])\leq V_R
\end{equation*}
holds for all $i$ and all $\lambda\in[0,i-R]$. In particular, 
$\vol(W_i\cap B_R^{M_i}(x))\leq V_R$ for all $i$ and $x\in M_i$.
\end{step}
The next key ingredient is the following uniform bound on the second
fundamental form $A^{W_L}$.
\begin{step}\label{step:SFF}
There is a constant $C_1>0$, depending only on the geometry of $(M,\bar g)$,
such that 
\begin{equation*}
\sup_{x\in W_L}|A^{W_L}(x)|^2\leq C_1 \ \  \ \mbox{for $L\geq0$.}
\end{equation*}
\end{step}
Step 2 is a consequence of \cite[Corollary 1.1]{SS}. See Appendix,
Section \ref{section:SFF} for more details.

\begin{step}\label{step3}
For each $R>0$ and $j=1,2,\ldots$, the sequence of hypersurfaces
$\Psi^{-1}_j(W_i^R)$ sub-converges smoothly locally as graphs as $i\to \infty$.
\end{step}
\begin{proof}[Proof of Step \ref{step3}]
We restrict our 
attention to the tail of the sequence $\{W^R_i\}_{i=1}^\infty$, where
$i\geq R+1$. This allows us to consider each $W^R_i$ and $W^{R+1}_i$
as hypersurfaces of $Y\times(-\infty,0]$ which is where we will show
the convergence. By rescaling the original metric $\bar g$, we will assume
that $\mathrm{inj}_g\geq1$ and the bounds
\begin{equation*}
  \begin{array}{c}
    \sup_{x\in B_1(y)}\left|\bar g_{ij}(x)-\delta_{ij}\right|\leq\mu_0,\quad
\sup_{x\in B_1(y)}\left|\frac{\p \bar g_{ij}}{\p x^k}(x)\right|\leq\mu_0
\end{array}
  \end{equation*}
hold for $1\leq i,j,k\leq n+1$ in geodesic normal 
coordinates centered about any 
$y\in Y\times(-\infty,0]$ where $\mu_0$
is the constant from Lemma \ref{lem:Riemgraphical}.
Let $r=\min(\frac{1}{24},\frac{1}{6\sqrt{20C_0}})$
where $C_0$ is the constant from Step \ref{step:SFF}. 
 
We cover $Y\times[-R,0]$ by a finite collection of open balls 
$\mathcal{U}=\{B_{r}(y_l)\}_{l=1}^N$. Notice that
each $B_{r}(y_l)\subset\subset Y\times[-R-1,0]$.
Consider a ball $B_{r}(y_l)$ in $\mathcal{U}$ with the property that
\begin{equation*}
W_i^{R+1}\cap B_r(y_l)\neq\emptyset
\end{equation*} 
for infinitely many $i$. Unless explicitly stated, we will continue
to denote all subsequences by 
$W_i^{R+1}$. Our next goal is to
show that the sequence of hypersurfaces
$\{W_i^R\cap B_r(y_l)\}_{i=1}^\infty$ sub-converges
smoothly locally as graphs.

We choose a subsequence of $W_i^{R+1}$ and points $x_i\in
W_i^{R+1}\cap B_r(y_l)$ which converge to some point
$x_\infty\in\overline{B_r(y_l)}$.  Now it will
be convenient to work in the tangent space
to the point $x_\infty$. We
use the short-hand notation $\phi=\exp_{x_\infty}^{\bar g}$
and let
\begin{equation*}
  B=\phi^{-1}(B_1(x_\infty))\subset T_{x_\infty}(Y\times[-L-1,0]).
\end{equation*}  
Consider the properly embedded hypersurfaces $\Sigma_i\subset B$
with base points $p_i\in \Sigma_i$, given by 
\begin{equation*}
\Sigma_i=\phi^{-1}(B_1(x_\infty)\cap W_i^R),\quad p_i=\phi^{-1}(x_i).
\end{equation*}
We also write $Z=\phi^{-1}(y_l)$
Since $W_i^R\subset M_i$ are minimal, $\Sigma_i$ are minimal
hypersurfaces in $B$ with
respect to the metric $\bar g_B=(\phi^{-1})^*(\bar g)$.

Notice that the choice of $r$ allows us to 
apply Corollary \ref{cor:schauder} to each $\Sigma_i\subset B$ at
$p_i$ with $s=3r$.  For each $i=1,2,\dots$, we obtain an open subset $U_i\subset
T_{p_i}\Sigma_i\cap B$, a unit normal
vector $\eta_i\perp T_{p_i}\Sigma_i$, and a function
$u_i:U_i\to\mathbb{R}$ satisfying the bounds (\ref{eq:schauder}) and
such that $\mathrm{graph}(u_i)=B^{\Sigma_i}_{6r}(p_i)$. 
Moreover, the connected component of $B^{\bar g_B}_{3r}(p_i)\cap\Sigma_i$ containing 
$x_0$ lies in $B^{\Sigma_i}_{6r}(p_i)$. 

We use compactness of $S^n$ and pass to a subsequence so that the vectors
$\eta_i$ converge to some vector $\eta_\infty\in S^n$.  Let
$P_\infty\subset T_{x_\infty}(Y\times[-L-1,0])$
be the hyperplane perpendicular to $\eta_\infty$.
For large enough $i$, we may translate and rotate the sets $U_i$ to
obtain open subsets $U_i^\prime\subset P_\infty$
and functions $u_i^\prime:U_i^\prime\to\mathbb{R}$ such that
\begin{enumerate}
\item $\mathrm{graph}(u_i^\prime)=B_{4r}^{\Sigma_i}(p_i)$;
\item the ball $B_{2r}^{P_\infty}(0)\subset U_i^{\prime}$;
\item for each $k\geq1$ and $\alpha\in(0,1)$ there is a constant $C^\prime>0$, depending 
only on $n$, $k$, $\alpha$, and the geometry of $g$, such that
\[
||u_i^\prime||_{C^{k,\alpha}(U^{\prime}_i)}\leq C^\prime,
\]
\end{enumerate}
see Fig. \ref{fig:04}.
In particular, writing $u_i^{\prime\prime}=u_i^\prime|_{B_{2r}^{P_\infty}(0)}$,
the sequence $\{u_i^{\prime\prime}\}_{i}$ is uniformly
bounded in $C^{k,\alpha}(B_{2r}^{P_\infty}(0))$.
Moreover, the connected component of $B_{2r}(p_i)\cap \Sigma_i$
containing $p_i$ is contained in $\mathrm{graph}(u_i^{\prime\prime})$.
It follows that $\Sigma_i^\prime$, the connected component of $B_r(Z)\cap\Sigma_i$
containing $p_i$, lies in $\mathrm{graph}(u_i^{\prime\prime})$.
\begin{figure}[!htb]
\begin{picture}(0,0)
\put(84,50){{\small $B_{2r}(0)$}}
\put(54,60){{\small $x^\prime$}}
\put(34,50){{\small \textcolor{blue}{$U^\prime_i$}}}
\put(14,83){{\small \textcolor{blue}{$U_i$}}}
\put(-2,128){{\small \textcolor{blue}{$\Sigma_i'$}}}
\put(-20,84){{\small \textcolor{blue}{$T_{p_i}\!\Sigma_i$}}}
\put(0,42){{\small \textcolor{red}{$P_{\infty}$}}}
\put(31,88){{\small $x$}}
\put(80,112){{\small $p_i$}}
\put(93,78){{\small $0$}}
\put(11,102){{\small \textcolor{blue}{$u_i\!(x)$}}}
\put(48,82){{\small \textcolor{red}{$u_i^\prime\!(x)$}}}
\put(88,10){{\small $B$}}
\end{picture}
\includegraphics[height=2.5in]{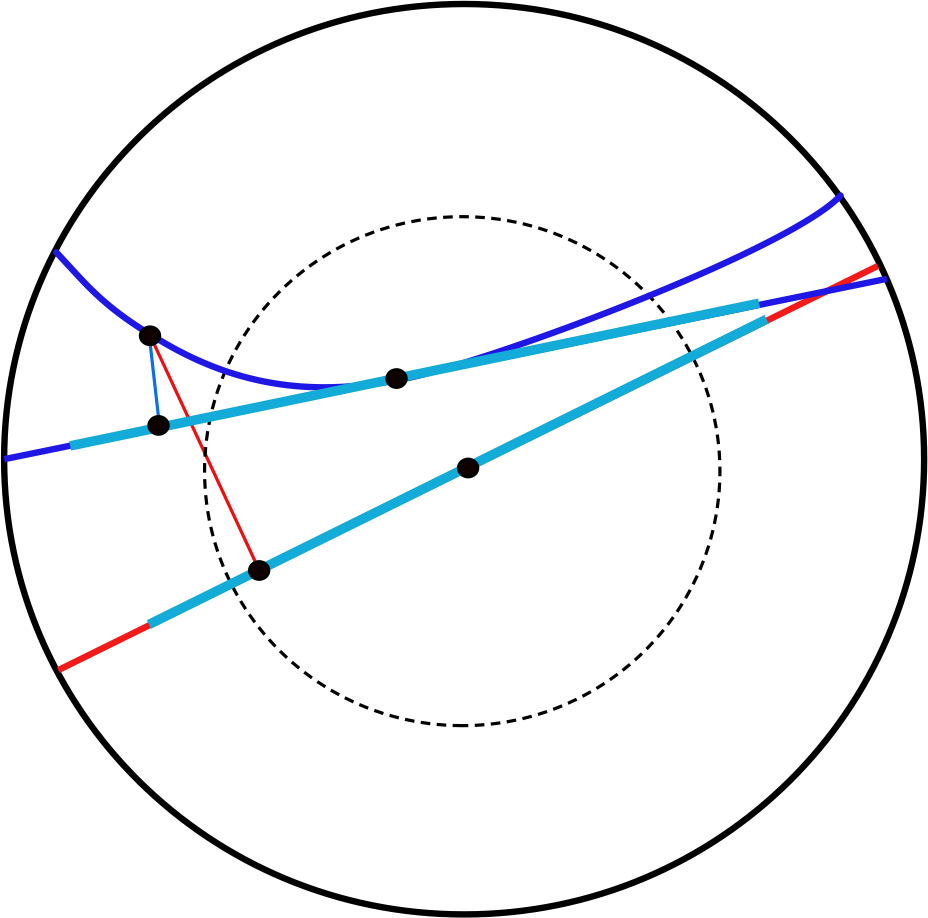}
\caption{The functions $u^\prime_i$ and hypersurfaces $\Sigma_i^\prime$}
\label{fig:04}
\end{figure} 

By Arzela-Ascoli, one can find a subsequence of $u_i^{\prime\prime}$ converging in
$C^k(B_{2r}^{P_\infty}(0))$ to a
function $u_\infty:B_{2r}^{P_\infty}(0)\to\mathbb{R}$. 
In particular, $u_\infty$ is a strong solution to the
minimal graph equation on $B_{2r}^{P_\infty}(0)$ with respect to 
$\bar g_B$ and $\Sigma_i^\prime$ converge as graphs to $\mathrm{graph}(u_\infty)$.
To summarize our current progress, the components of 
$W_i^{R+1}\cap B_r(y_l)$ containing $x_i$
sub-converge smoothly to $\phi(\mathrm{graph}(u_\infty))$.
This finishes our work with the hypersurfaces $\Sigma_i^\prime$.

Now suppose that there is a second sequence of connected components
within $W_i^{R+1}\cap B_r(y_l)$.  
We can repeat the above process to obtain a second limiting hypersurface.
Observe that the number of components of $W_i^{R+1}\cap B_r(y_l)$
uniformly bounded in $i$.  
Indeed, using the notation above, for any component $\bar{\Sigma_i}\subset W_i^{R+1}\cap B_r(y_l)$, we have
\begin{equation*}
\begin{array}{c}
  \vol_{\bar g_B}(\Sigma_i^\prime)\geq \vol_{\bar g_B}(B_r^{P_\infty}(0)),
\end{array}
\end{equation*}
which is uniformly bounded below in terms of $r$ and the geometry of $g$.
However, Step \ref{step:volume} implies that $\vol(W_i^R\cap B_r(y_l))$ is
bounded above uniformly in $i$ so the number of connected components
$W_i^R\cap B_r(y_l)$ is uniformly bounded in $i$. 
Hence the above process terminates after finitely many iterations.
We conclude that the sequence $\{W_i^R\cap B_{r}(y_l)\}_{i=1}^\infty$
sub-converges smoothly locally as graphs to a minimal hypersurface
$\Sigma_{\infty,l}$.

Now, restricting to this subsequence, we turn our attention to another
ball $B_{r}(y_{l^\prime})$ in the cover $\mathcal{U}$. We repeat the above
argument to obtain a further subsequence and limiting minimal
hypersurface $\Sigma_{\infty,l^\prime}$. Repeating this process for each
element of $\mathcal{U}$ produces a subsequence
converging to a minimal hypersurface
$W_\infty^R=\bigcup_{l=0}^N \Sigma_{\infty,l}$ smoothly
locally as graphs. This completes the proof of Claim \ref{claim3}, and
consequently, the proof of Main Lemma. 
\end{proof}

\section{Proof of Theorem 5}\label{sec4}
\subsection{Positive conformal bordism}
In order to prove Theorem \ref{thm5}, we have to use fundamental facts
relating conformal geometry and psc-bordism. We briefly
recall necessary results, following the conventions in \cite{AB}.  Let
$Y$ be a compact closed manifold with $\dim Y =n$ given together with
a conformal class $C$ of Riemannian metrics.  Then the \emph{Yamabe
  constant} of $(Y,C)$ is defined as
\begin{equation*}
Y(Y;C)=\inf_{g\in C}\frac{\int_Y R_{g}d\mu_{g}}{\vol_g(Y)^{\frac{n-2}{n}}}.
\end{equation*}
We say that a conformal class $C$ is \emph{positive} if $Y(Y;C)>0$. It
is well-known that $C$ is positive if and only if there exists a
psc-metric $g\in C$. 

Now let $Z: Y_0 \rightsquigarrow Y_1$ be a bordism between closed
manifolds $Y_0$ and $Y_1$. Suppose we are given 
conformal classes $C_0$ and $C_1$ on $Y_0$ and $Y_1$, respectively.
Let $\bar C$ be a conformal class on $Z$, such
that $\bar C|_{Y_0}=C_0$ and $\bar C|_{Y_1}=C_1$, i.e. $\p \bar C=
C_0\sqcup C_1$.  Denote by $\bar C^0 =\{\bar g\in\bar C\colon H_{\bar
  g}\equiv0\}$ the subclass of those metrics with
vanishing mean curvature of the boundary.
Then the \emph{relative Yamabe constant} of $((Z,\bar C), (Y_0\sqcup
Y_1,C_0\sqcup C_1))$ is defined as
\begin{equation*}
Y_{\bar C}(Z,Y_0\sqcup Y_1;C_0\sqcup C_1)=\inf_{\bar g\in\bar C^0} \frac{\int_Z
  R_{\bar g}d\mu_{\bar g}}{\vol_{\bar g}(Z)^{\frac{n-2}{n}}}.
\end{equation*} 
This gives the {\emph{relative Yamabe invariant}} 
\begin{equation*}
Y(Z,Y_0\sqcup Y_1;C_0\sqcup C_1) =\sup_{\bar C,\;\p \bar C=C_0\sqcup C_1}Y_{\bar
  C}(Z,Y_0\sqcup Y_1;C_0\sqcup C_1) .
\end{equation*} 
Now we assume that the conformal classes $C_0$ and $C_1$ are positive.
Then we say that positive conformal manifolds $(Y_0,C_0)$ and
$(Y_1,C_1)$ are \emph{positive-conformally bordant} if there exists a
conformal manifold $(Z,\bar C)$ and a bordism $Z: Y_0 \rightsquigarrow
Y_1$ between $Y_0$ and $Y_1$ such that $\p \bar C= C_0\sqcup C_1$
and $Y_{\bar C}(Z,Y_0\sqcup Y_1;C_0\sqcup
C_1)>0$. In this case, we write $(Z,\bar C) : (Y_0,C_0) \rightsquigarrow
(Y_1,C_1)$.
We need the following result which relates the above notions to psc-bordisms.
\begin{theorem}\cite[Corollary B]{AB}\label{thm:AB}
  Let $Y_0$ and $Y_1$ be closed manifolds of dimension $n\geq 3$, $Z:
  Y_0 \rightsquigarrow Y_1$ be a bordism between $Y_0$ and $Y_1$, and
  $g_0$ and $g_1$ be psc-metrics on $Y_0$ and $Y_1$, respectively. Then
  $Y(Z,Y_0\sqcup Y_1;[g_0]\sqcup [g_1])>0 $ if and only if the
  boundary metric $g_0\sqcup g_1$ on $Y_0\sqcup Y_1$ may be extended
  to a psc-metric $\bar g$ on $Z$ such that $\bar g= g_j+dt^2$ near
  $Y_j$ for $j=0,1$.
\end{theorem}
\subsection{Long collars}\label{sec:longcollars}
We are ready to prove Theorem \ref{thm5} for $n\leq6$. The adjustments
required to adapt the following proof to the case $n=7$
are provided in Appendix \ref{sec:dim8}.

Let $(Y_0,g_0,\gamma_0)$ and $(Y_1,g_1,\gamma_1)$ be the manifolds
from Theorem \ref{thm5} and 
let $\alpha_0\in H_{n-1}(Y_0; \mathbf{Z})$ and $\alpha_1\in
H_{n-1}(Y_1; \mathbf{Z})$ be the classes Poincar\`e dual to $\gamma_0$
and $\gamma_1$, respectively. It is convenient to use the notation
\footnote{\ \ Here we emphasize a proper orientation on $Y_0$ and $Y_1$}
$Y=Y_0\sqcup - Y_1$ and
\begin{equation*}
\alpha=(\iota_0)_*\alpha_0-(\iota_1)_*\alpha_1\in H_{n-1}(Y;\mathbf{Z}) ,
\end{equation*}  
where $\iota_j:Y_j\rh Y$ is the inclusion map for $j=0,1$.  Then we
consider hypersurfaces $X_0\subset Y_0$ and $X_1\subset Y_0$ 
which are homologically volume minimizing representatives of 
the classes $\alpha_0$ and
$-\alpha_1$.  The existence of such smooth $X_0$ and $X_1$ is guaranteed 
in this range of dimensions, see \cite{SY79}. Notice that, by a small 
conformal change which does not effect the assumptions on 
$(Y_j,g_j,\gamma_j)$,
we may assume that $X_j$ is the only representative of $\alpha_j$ with
minimal volume for $j=0,1$, see \cite[Lemma 1.3]{Smale93}.
We write $(X,h_X)$ for the Riemannian manifold
$(X_0\sqcup X_1, g_0|_{X_0}\sqcup g_1|_{X_1})$.

Now we choose a psc-bordism $(Z,\bar g,\bar \gamma):
(Y_0,g_0,\gamma_0) \rightsquigarrow (Y_1,g_1,\gamma_1)$. We will 
use $(Z,\bar g,\bar \gamma)$ to construct a psc-bordism which
satisfies the conclusion of Theorem \ref{thm5}. We denote by $\bar
\alpha\in H_{n}(Z;\mathbf{Z})$ the homology class Poincar\`e dual to
$\bar \gamma$. Then $\p\bar \alpha =\alpha$, see Lemma \ref{lemma01}.

Now for each $i=1,2,\ldots$, we consider the $i$-collaring of the bordism $(Z,\bar g, \bar
\gamma)$, denoted by $(Z_i,\bar g_i,
\bar \gamma_i)$, as in Section \ref{sec3-2}. By Theorem \ref{thm4}, there
exists properly embedded hypersurfaces $W_i\subset Z_i$ which are
homologically volume minimizing and represents $\bar\alpha_i$.
The restrictions of $\bar g_i$ to $W_i$ and $\p W_i$ are denoted by 
$\bar h_i$ and $h_i$, respectively. 

In preparation to apply Main Lemma, we fix basepoints $x_j\in X_j$ for
each $j=0,1$ and set $\mathsf{S}=\{x_0,x_1\}\subset X$.  Naturally,
the set $\mathsf{S}$ is identified with the subsets $\mathsf{S}_i$
in $(X\times\{0\})\subset \p Z_i$ for $i=1,2,\ldots$ and with
$\mathsf{S}_\infty$ in the boundary of the cylinder
$(X\times\{0\})\subset(Y\times(-\infty,0])$.  According to Main Lemma we may
  find a subsequence $\{a_i\}_{i=1}^\infty$ such that 
\begin{equation*}
(Z_{a_i},W_{a_i},\bar g_{a_i},\mathsf{S}_{a_i})\longrightarrow
(Y\times(-\infty,0],X\times(-\infty,0],g+dt^2,\mathsf{S}_\infty)
\end{equation*}
smoothly as $i\to\infty$ and the Riemannian manifolds $(\p
W_{a_i},h_{a_i})$ converge to $(X,h_X)$ in the smooth Cheeger-Gromov
topology as $i\to\infty$.
\begin{rem}
We note that the manifolds $(\p W_{a_i},h_{a_i})$, $(X,h_X)$ are
compact and so there is no need to specify base points for this convergence.
\end{rem}  
The following is a special case of a much more general fact on 
the behavior of elliptic eigenvalue problems under
smooth Cheeger-Gromov convergence (see \cite{BM}).
\begin{lemma}\label{lem:thm5}
Let $\{(M_i,g'_i)\}_{i=1}^\infty$ be a sequence of compact Riemannian manifolds
smoothly converging to a compact Riemannian manifold
$(M_\infty,g'_\infty)$ in the Cheeger-Gromov sense.  If
$Y(M_\infty;[g'_\infty])>0$, then, upon passing to a subsequence, $Y(M_i;[g'_i])>0$ for all
sufficiently large $i$.
\end{lemma}
\begin{proof}
For each $i=1,2,\ldots$, we denote by $\lambda_{1,i}=\lambda_1(L_{g'_i})$
the principal eigenvalue of the conformal Laplacian on $(M_i,g'_i)$.
Let $\phi_i\in C^\infty(M_i)$ be the eigenfunction satisfying
\begin{equation}\label{eq:thm5}
L_{g'_i}\phi_i=\lambda_{1,i}\phi_i, \quad\quad \sup_{M_i}\phi_i=1.
\end{equation}
Since $\{(M_i,g'_i)\}_{i=1}^\infty$ is converging in the Cheeger-Gromov topology to a
compact manifold, the coefficients of the operator $L_{g'_i}$ are
bounded in the $C^1$-norm uniformly in $i$. In particular, there is a
constant $C_1>0$, independent of $i$, such that $|R_{g'_i}|\leq C_1$
on $M_i$. An obvious estimate on the Rayleigh quotient
(\ref{eq:Rayleigh}) shows that the sequence
$\{\lambda_{1,i}\}_{i=1}^\infty$ is uniformly bounded above and below.

This allows us to apply the Schauder estimate Theorem \ref{thm:schauder}
to $\phi_i$ uniformly in $i$. Using
Arzel\'a-Ascoli, we can find a subsequence, still denoted by $\{(M_i,g'_i)\}_{i=1}^\infty$,
$\{\phi_i\}_{i=1}^\infty$, and $\{\lambda_{1,i}\}_{i=1}^\infty$, a function 
$\phi_\infty\in C^\infty(M_\infty)$, and a number $\lambda_{1,\infty}$ such that
\begin{equation*}
\phi_i\to\phi_\infty \quad \lambda_{1,i}\to\lambda_{1,\infty}
\end{equation*}
where the former convergence is in the $C^{2,\alpha}$-topology.
This allows us to take the limit of equation (\ref{eq:thm5}) as $i\to\infty$.
Namely, $\phi_\infty$ is a non-zero solution of the equation
\begin{equation*}
L_{g_\infty}\phi_\infty=\lambda_{1,\infty} \phi_\infty
\end{equation*}
and so $\lambda_{1,\infty}
\geq \lambda_1(L_{g_\infty})$. On the other hand, we have assumed that
$\lambda_1(L_{g_\infty})>0$.  Hence $\lambda_{1,i}>0$ for all sufficiently
large $i$.
\end{proof}
Now we return to the proof of Theorem \ref{thm5}.  Since $X$ is a
stable minimal hypersurface of $Y$ with trivial normal bundle, Theorem
\ref{thm1} implies that $Y(X,[g_X])>0$. Now we may apply Lemma
\ref{lem:thm5} to find $Y(\p W_{a_i},[h_{a_i}])>0$ for sufficiently
large $i$.  Fix such an $i$
and let $h'_{a_i}\in[h_{a_i}]$ be a psc metric on $\p W_{a_i}$.
Since each $W_{a_i}$ is a stable minimal hypersurface with
free boundary and trivial normal bundle, Theorem \ref{thm3} states
that $Y(W_{a_i},\p W_{a_i};[\bar h_{a_i}])>0$ for all
$i\in\mathbb{N}$. 
Finally, we use Theorem \ref{thm:AB} to find a psc-metric $\tilde
h_{a_i}$ on $W_{a_i}$ which restricts to $h'_{a_i}+dt^2$ near $\p
W_{a_i}$.  This completes the proof of Theorem \ref{thm5} for $n\leq6$.

\appendix
\section{}
The main goal here is to provide technical details we used in
the main body of the paper. In Section \ref{sec:graphrep}, we recall
relevant facts on the minimal graph equation and provide the Schauder
estimates we use in the proof of Main Lemma.  Section \ref{sec:thm4}
is dedicated to Theorem \ref{thm4}. Here we recall necessary results
on currents and state well-known facts on their compactness and
regularity, adapted to our setting.  Section \ref{sec:doubling}
describes a simple doubling method which is a convenient technical
tool in the remaining sections. In Section
\ref{section:SFF}, we justify Step 2 from the proof of Claim
\ref{claim3}. In Section \ref{sec:dim8}, we discuss regularity issues
in dimension $8$ and prove Theorem \ref{thm5} for $n=7$.

\subsection{The minimal graph equation}
\label{sec:graphrep}
This section is concerned with local properties of hypersurfaces in
Riemannian manifolds. Throughout this section we will consider the
unit ball in Euclidian space $B=B_1(0)\subset\mathbb{R}^{n+1}$
equipped with a Riemannian metric $g$ and a hypersurface
$\Sigma^n\subset B$.  The balls of radius $s>0$ centered at
$x\in \Sigma$ induced by $g$ and $g|_\Sigma$ are denoted by
$B^g_s(x)\subset B$ and $B^\Sigma_s(x)\subset \Sigma$, respectively.
Assume there is a point $x_0\in\Sigma\cap B_{1/4}(0)$.

The following straight-forward Riemannian version of \cite[Lemma
2.4]{CM} allows us to consider $\Sigma$ locally as a graph over
$T_{x_0}\Sigma$.
\begin{lemma}\label{lem:Riemgraphical}
There is a constant $\mu_0>0$ so that if $g$ satisfies
\begin{equation}\label{eq:mu0}
\begin{array}{c}
{\displaystyle \sup_{x\in B}}\left| g_{ij}(x)-\delta_{ij}\right|\leq\mu_0,\quad
{\displaystyle \sup_{x\in B}}\left|\frac{\p  g_{ij}}{\p x^k}(x)\right|\leq\mu_0
\end{array}
\end{equation}
for $1\leq i,j,k\leq n+1$ in standard Euclidian coordinates, then the
following holds:
\noindent
If $s>0$ satisfies
\begin{equation*}
\begin{array}{c}
 \mathrm{dist}^\Sigma(x_0,\p\Sigma)\geq3s,\quad 
\sup_{\Sigma}|A_g|^2\leq\frac{1}{20 s^2},
\end{array}
\end{equation*}
then there is an open subset $U\subset
T_{x_0}\Sigma\subset\mathbb{R}^{n+1}$, a unit vector $\eta$ normal to
$T_{x_0}\Sigma$, and a function $u:U\to\mathbb{R}$ such that
\begin{enumerate}
\item $\mathrm{graph}(u)=B_{2s}^\Sigma(x_0)$;
\item $|\nabla u|\leq1$ and $|\nabla\nabla
  u|\leq\frac{1}{s\sqrt{2}}$ hold pointwise.
\end{enumerate}
Moreover, the connected component of $B_s^{g}(x_0)\cap\Sigma$ containing $x_0$
lies in $B_{2s}^\Sigma(x_0)$.
\end{lemma}

Now we will give a useful expression for the 
mean curvature of a graph.
Let $U\subset\mathbb{R}^n$ be an
open set with standard coordinates $x'=(x^1,\ldots,x^n)$ 
and let $g$ be a Riemannian metric on
$U\times\mathbb{R}\subset\mathbb{R}^{n+1}$. For a function
$u:U\to\mathbb{R}$, consider its graph
\[
\mathrm{graph}(u)=\{(x',u(x'))\in\mathbb{R}^{n+1}\colon x'\in U\}.
\]
For $i=1,\ldots,n$, we have the tangential vector fields
$E_i=\frac{\p}{\p x^i}+\frac{\p u}{\p x^i}\frac{\p}{\p x^{n+1}}$ and the
upward-pointing unit vector field $\nu$ normal to $\mathrm{graph}(u)$.
Writing $h_{ij}=g(E_i,E_j)$ for the restriction metric, the mean 
curvature of $\mathrm{graph}(u)$ can be written
\begin{equation}\label{eq:mean}
\begin{array}{lcl}
  H_g&= & h^{ij}g(\nu,\nabla_{E_i}E_j)\\
  \\
  &= & \left(g^{ij}-\frac{\nabla^i u\nabla^j u}{1+|\nabla u|^2}\right)\Big{[}\frac{\p^2u}{\p x_i\p x_j}+\Gamma^{n+1}_{ij}+\frac{\p u}{\p x_i}\Gamma^{n+1}_{n+1\; j}+\frac{\p u}{\p x_j}\Gamma^{n+1}_{n+1\; i}+\frac{\p u}{\p x_i} \frac{\p u}{\p x_j}\Gamma^{n+1}_{n+1\; n+1} \\
  \\
& & \quad -\frac{\p u}{\p x_r}\left(\Gamma^r_{ij}+\frac{\p u}{\p x_i}\Gamma^r_{n+1\; j}+\frac{\p u}{\p x_j}\Gamma^r_{i\; n+1}+\frac{\p u}{\p x_i}\frac{\p u}{\p x_j}\Gamma^r_{n+1\; n+1}\right)\Big{]},
\end{array}
\end{equation}
see \cite[Section 7.1]{CM} for a detailed exposition in the
$3$-dimensional case.

Next, we will state a general version of the Schauder estimates 
for elliptic operators on Euclidian space. It is 
applied to the geometric setting in Section \ref{sec3}.
\begin{theorem}\cite[Corollary 6.3]{GT}\label{thm:schauder}
Let $U\subset\mathbb{R}^n$ be an open set and let $\alpha\in(0,1)$. Suppose 
$u\in C^{2,\alpha}(U)$ satisfies a uniformly elliptic equation
\[
Lu=a^{ij}(x)u_{ij}+b^i(x)u_i+c(x)u=0
\]
with $a^{ij},b^i,c\in C^\alpha(U)$ and ellipticity constant $\lambda>0$. 
If $U'\subset\subset U$ 
with $\mathrm{dist}^{U}(U',\p /u)= d$, then there
is a constant $C>0$, depending on $d,\lambda,||a^{ij}||_{C^\alpha(U)},
||b^i||_{C^\alpha(U)},||c||_{C^\alpha(U)},n$, and $\alpha$, such that
\begin{equation}\label{eq:GTest}
||u||_{C^{2,\alpha}(U')}\leq C||u||_{C^0(U)}.
\end{equation}
\end{theorem}

\begin{cor}\label{cor:schauder}
Suppose the unit ball $B=B_1(0)\subset \mathbb{R}^{n+1}$ 
is equipped with a Riemannian metric $g$ satisfying 
\begin{equation*}
\sup_{x\in B}\left| g_{ij}(x)-\delta_{ij}\right|\leq\mu_0,\quad
\sup_{x\in B}\left|\frac{\p  g_{ij}}{\p x^k}(x)\right|\leq\mu_0
\end{equation*}
in Euclidian coordinates for all $1\leq i,j,k\leq n+1$ where $\mu_0$
is the constant from Lemma \ref{lem:Riemgraphical}. 
Let $C>0$ be given and set $r=\min(\frac18,\frac{1}{\sqrt{80C}})$.
Assume that
$\Sigma\subset B$ is a properly embedded minimal hypersurface with 
respect to $g$ such that 
$\sup_B|A^g|^2\leq C$ and there is a point $x_0\in B_{r}(0)\cap \Sigma$.
Then there is a smooth function $u:U\to\mathbb{R}$ on $U\subset T_{x_0}\Sigma$
and a unit normal vector to $T_{x_0}\Sigma$ such that
\begin{enumerate}
\item $\mathrm{graph}(u)=B^\Sigma_{2r}(x_0)$;
\item $|\nabla u|\leq1$ and $|\nabla\nabla
  u|\leq\frac{1}{s\sqrt{2}}$ hold pointwise;
\item for each $k\geq1$ and $\alpha\in(0,1)$ there is a constant $C^\prime>0$, 
depending only on 
$n,k,\alpha,$ and $||g||_{C^{k,\alpha}(B)}$, such that
\begin{equation}\label{eq:schauder}
||u||_{C^{k,\alpha}(U)}\leq C^\prime.
\end{equation}
\end{enumerate}
Moreover, the connected component of $B_r(x_0)\cap\Sigma$ containing $x_0$
is contained in $B^\Sigma_{2r}(x_0)$.
\end{cor}

\begin{proof}
The choice of radius $r$ allows us to apply Lemma \ref{lem:Riemgraphical}
to obtain an open subset $U\subset
T_{x_0}\Sigma\subset\mathbb{R}^{n+1}$, a unit vector $\eta$ normal to
$T_{x_0}\Sigma$, and a smooth function $u:U\to\mathbb{R}$ such that
 $\mathrm{graph}(u)=B_{2s}^\Sigma(x_0)$, $|\nabla u|\leq1$, and $|\nabla\nabla
u|\leq\frac{1}{s\sqrt{2}}$ on $U$. 
Since $\Sigma$ is minimal, $u$ solves equation $H=0$.
Now since $||u||_{C^{1,\alpha}(U)}$ is bounded for any fixed $\alpha\in(0,1)$, 
one can inspect the expression \ref{eq:mean} to 
see that $u$ solves a linear elliptic equation with coefficients 
bounded in $C^\alpha$ in terms of $\mu_0$ and $r$. 
This allows us to apply Theorem \ref{thm:schauder} to obtain the estimate 
$||u||_{C^{2,\alpha}(U')}\leq C||u||_{C^0(U)}$ for some $C>0$ depending
only on $\mu_0$ and $r$. Standard elliptic estimates \cite[Section 6]{GT}
give a similar estimate in the $C^{k,\alpha}$-norm for any $k$.
\end{proof}
\subsection{Details on Theorem \ref{thm4}}\label{sec:thm4} 
Let us recall some basic notions from theory of
integer multiplicity currents.  The main reference for this material is
\cite[Chapter 4]{F}.

For an open subset $U\subset\mathbb{R}^{n+k}$, let $\Omega^n(U)$
denote the space of all $n$-forms on $\mathbb{R}^{n+k}$ with compact
support in $U$.  An {\emph{$n$-current}} on $U$ is a
continuous linear functional $T:\Omega^n(U)\to \mathbb{R}$ and 
collection of such $T$ for a vector space $\mathcal{D}_{n}(U)$.  The
{\emph{boundary}} of an $n$-current $T$ is the $(n-1)$-current 
$\p T$ defined by
\begin{equation*}
(\p T)(\omega)=T(d\omega),\quad\quad \omega\in\Omega^{n-1}(U).
\end{equation*}
The \emph{mass} of 
$T\in \mathcal{D}_{n}(U)$ is given by
$\mathbf{M}(T)=\sup\{T(\omega):\omega\in\Omega^{n}(U),|\omega|\leq1\}$.
For example, if $T$ is given by integration along a smooth oriented 
submanifold $M$, then $\mathbf{M}(T)=\vol(M)$.
 
Let $\mathcal{H}^n$ denote the $n$-dimensional Hausdorff measure 
on $\mathbb{R}^{n+k}$.  
A current $T\in\mathcal{D}_n(U)$ is called 
\emph{integer multiplicity rectifiable} (or simply \emph{rectifiable}) if it takes the form
\begin{equation}\label{eq:rectifiable}
T(\omega)=\int_M\omega(\xi(x))\theta(x) d\mathcal{H}^n(x),\quad\quad\omega\in\Omega^n(U), \ \ \ \mbox{where}
\end{equation}
\begin{enumerate}
\item $M\subset U$ is $\mathcal{H}^n$-measurable and countably $n$-rectifiable, see \cite[Section 3.2.14]{F};
\item $\theta:M\to\mathbf{Z}$ is locally $\mathcal{H}^n$-integrable;
\item for $\mathcal{H}^n$-almost every $x\in M$,
$\xi:M\to\Lambda^nT\mathbb{R}^{n+k}$ takes the form
$\xi(x)=e_1\wedge\ldots\wedge e_n$ where $\{e_i\}_{i=1}^n$ form
an orthonormal basis for the approximate tangent space $T_xM$,
see \cite[Section 3.2.16]{F}.
\end{enumerate}
\begin{rem}
The above definition of integer multiplicity rectifiable currents can
also be extended to Riemannian manifolds $(M,g)$ -- one defines the
mass of a current using the Hausdorff measure given by the metric
$g$.
\end{rem}
\noindent
The {\emph{regular}} set $\mathrm{reg}(T)$ of a rectifiable
$n$-current $T$ is given by the set of points $x\in \mathrm{spt}(T)$
for which there exists an oriented $n$-dimensional oriented
$C^1$-submanifold $M\subset U$, $r>0$, and $m\in\mathbf{Z}$ satisfying
\[
T|_{B_r(x)}(\omega)=m\cdot\int_{M\cap B_r(x)}\omega,\quad\quad\forall\omega\in\Omega^n(U).
\] 
The {\emph{singular}} set $\mathrm{sing}(T)$ is given by
$\spt(T)\setminus\mathrm{reg}(T)$.  The abelian group of
$n$-dimensional \emph{integral flat chains} on $U$ is given by
\[
\mathcal{F}_n(U)=\{R+\partial S\colon R\in\mathcal{D}_n(U)\text{ and }
	S\in\mathcal{D}_{n+1}(U)\text{ are rectifiable}\}.
\]
Now we consider subsets $B\subset A\subset U$. We have the group of 
{\emph{integral flat cycles}}
\[
\mathcal{C}_n(A,B)=\{T\in\mathcal{F}_n(U)\colon
	\mathrm{spt}(T)\subset A,\mathrm{spt}(\p T)\subset B,\text{ or }n=0\}
\]
and the subgroup of {\emph{integral flat boundaries}}
\[
\mathcal{B}_n(A,B)=\{T+\p S\colon T\in\mathcal{F}_n(U),\spt(T)\subset B,
	S\in\mathcal{F}_{n+1}(U),\spt(S)\subset A\}.
\]
The quotient groups
$
\mathbf{H}_n(A,B)=\mathcal{C}_n(A,B)/\mathcal{B}_n(A,B)
$
are the $n$-dimensional {\emph{integral current homology groups}}.

There is a natural transformation between the integral singular homology
functor and the integral current homology functor which induces an isomorphism
$H_n(A,B;\mathbf{Z})\cong\mathbf{H}_n(A,B)$ in the category
of local Lipschitz neighborhood retracts, see \cite[Section 4.4.1]{F}.
This isomorphism can be combined with a basic compactness result
for rectifiable currents to find volume minimizing representatives 
of homology classes.
\begin{lemma}\label{lem:mincurrent}
Let $(M,\bar g)$ be a compact $(n+1)$-dimensional Riemannian manifold with boundary
and consider an integral homology class 
$\alpha\in H_n(M,\p M;\mathbf{Z})$. Let $\tilde{\alpha}\in\mathbb{H}_n(M,\p M)$
be the image of $\alpha$ under the isomorphism 
$H_n(M,\p M;\mathbf{Z})\to \mathbf{H}_n(M,\p M)$. Then there exists a homologically 
volume minimizing integer multiplicity rectifiable current $T\in\tilde\alpha$.
\end{lemma}
\begin{proof}
By the Nash embedding theorem there is an isometric embedding 
$\iota:M\to\mathbb{R}^{n+k}$
for some sufficiently large $k$. Let $\hat{M}$ be the image of this embedding and set 
$\hat \alpha=\iota_*\tilde\alpha\in\mathbf{H}_n(\hat M,\p\hat M)$. 
Applying the compactness result in \cite[Section 5.1.6]{F},
we obtain a homologically volume minimizing current 
$\hat T\in \mathcal{C}_n(\hat M,\p \hat M)$ representing $\hat \alpha$. Since $\iota$ is an
isometry, $(\iota^{-1})_*\hat T$ is the desired current.
\end{proof}
Since Lemma \ref{lem:mincurrent} guarantees the existence of 
homologically volume minimizing
representative for the homology class $\alpha$ from the hypothesis of Theorem 
\ref{thm4}, the final ingredient is regularity theory for volume minimizing 
rectifiable currents with free boundary.
The following is a regularity theorem 
due to M. Gr\"unter \cite[Theorem 4.7]{G} adapted to the context of an ambient Riemannian
metric. See \cite{L,FL,SSY} for Riemannian adaptations of similar results.
\begin{theorem}\label{thm:reg}
Let $S\subset\mathbb{R}^{n+1}$ be an $n$-dimensional smooth submanifold,
$U\subset\mathbb{R}^{n+1}$ an open set with 
$\partial S\cap U=\emptyset$, and $g$ a Riemannian metric on $U$ with 
bounded injectivity radius and sectional curvature.
Suppose $T\in\mathcal{F}_n(U)$ with $\spt(\p T)\subset S$ satisfies
$\mathbf{M}_g(T)\leq\mathbf{M}_g(T+R)$ 
for all open $W\subset\subset U$ and all
$R\in\mathcal{F}_n(U)$ with $\spt(R)\subset W$ and $\spt(\p R)\subset
S$. Then we have
\begin{itemize}
\item $\mathrm{sing}(T)=\emptyset$ if $n\leq 6$
\item $\mathrm{sing}(T)$ is discrete for $n=7$
\item $\mathrm{dim}_{\mathcal{H}}(\mathrm{sing}(T))\leq n-7$ if $n>7$
\end{itemize}
where $\mathrm{dim}_{\mathcal{H}}(A)$ denotes the Hausdorff dimension of a subset $A\subset U$.
\end{theorem}
We will briefly explain how Theorem \ref{thm4} follows from Theorem \ref{thm:reg}.
Let $T$ be the volume minimizing representative of
$\bar\alpha$ from Theorem \ref{thm4}.
For a point $x\in\mathrm{spt}(T)$, set $\phi=\exp_x^{\bar g}$ and consider
\begin{equation*}
U=\phi^{-1}(B^{\bar g}_{r'}(x))\subset T_x M,\quad S=\phi^{-1}(\p M\cap B^{\bar g}_{r}(x)),
\end{equation*}
\begin{equation*}
T'=(\phi^{-1})_*T\in\mathcal{D}_n(U),\quad g=(\phi^{-1})_*\bar g,
\end{equation*}
where $0<r'<r\leq\mathrm{inj}(\bar g)$. By Theorem \ref{thm:reg}, the
singular set of $T'$ is empty and so 
there is a neighborhood $V$ of $0\in U$ such that $T'|_V$ is given by
an integer multiple of integration along a $C^1$-submanifold $M\subset V$.
Locally, $M$ can be written as the graph of a $C^1$-function which
weakly solves the minimal surface equation. Standard elliptic PDE methods imply
that $M$ is smooth, see, for instance the proof of Lemma \ref{lem:double} below.

\subsection{Doubling minimal hypersurfaces with free boundary}
\label{sec:doubling}
In this section we consider the reflection of a free boundary stable
minimal hypersurface over its boundary. To fix the setting, let
$(M,\bar g)$ be an $(n+1)$-dimensional compact oriented Riemannian
manifold with boundary $\p M$ and restriction metric $g=\bar g|_{\p
M}$.  Assume that there is a neighborhood of the boundary on which
$\bar g=g_{\p M}+dt^2$.  The {\emph{double}} of $(M,\bar g)$ is the
smooth closed manifold $M_\mathcal{D}$ given by
$M_\mathcal{D}=M\cup_{\p M}(-M).$ 
Notice that the double $M_\mathcal{D}$ comes
equipped with an involution $\iota:M_\mathcal{D}\to M_\mathcal{D}$
which interchanges the two copies of $M$ and fixes the doubling locus
$\p M\subset M_\mathcal{D}$. Since $\bar g$ splits as a product near
the boundary, one can also form the smooth doubling of $\bar g$,
denoted by $\bar g_\mathcal{D}$, by setting $\bar g_\mathcal{D}=\bar
g$ on $M$ and $\bar g_\mathcal{D}=\iota_*\bar g$ on $-M$.
\begin{lemma}\label{lem:double}
Let $(M,\bar g)$ be a compact oriented Riemannian manifold with boundary 
with $\bar g=g+dt^2$ near $\p M$.
If $\Sigma\subset M$ be a properly embedded minimal hypersurface with free boundary,
then double of $\Sigma$, given by $\Sigma_{\mathcal{D}}=
\Sigma\cup_{\p\Sigma}\iota(\Sigma)$ is a smooth 
minimal hypersurface of $(M_\mathcal{D},\bar g_\mathcal{D})$.
Moreover, if $\Sigma$ is stable, then so is $\Sigma_\mathcal{D}$.
\end{lemma}
\begin{proof}
First, we will show that $\Sigma_\mathcal{D}$ is a smooth hypersurface.
Clearly, $\Sigma_\mathcal{D}$ is smooth away from the doubling locus 
$\p\Sigma\subset M_\mathcal{D}$.
Let $x_0\in\p\Sigma$ and let $r>0$ be less than the injectivity radius
of $\bar g_{\mathcal{D}}$.  
Set $\phi=\exp_{x_0}^{\bar g_\mathcal{D}}$ and consider 
\[
\hat\Sigma=\phi^{-1}
	(\Sigma\cap B_{r}(x_0)),\quad
\hat\Sigma_{\mathcal{D}}=\phi^{-1}
	(\Sigma_{\mathcal{D}}\cap B_{r}(x_0)),
\quad \hat g=\phi^*\bar g_\mathcal{D}
\]
and $\nu$, the unit normal vector field to $\hat\Sigma$ with respect to $\hat g$.
Evidently, $\hat\Sigma$ is a minimal hypersurface in $T_{x_0}M_\mathcal{D}$
with free boundary contained in $T_{x_0}\p M\subset T_{x_0}M_\mathcal{D}$
with respect to $\hat g$.
We choose an orthonormal basis for $T_{x_0}M_\mathcal{D}$
so that, writing $x\in T_{x_0}M$ as $(x^1,\ldots,x^{n+1})$ in this basis,
\begin{enumerate}
\item $T_{x_0}\p \hat\Sigma=\{(x^1,\ldots,x^{n-1},0,0)\}$;
\item $T_{x_0}\hat\Sigma=\{(x^1,\ldots,x^n,0)\}$;
\item $T_{x_0}\p M=\{(x^1,\ldots,x^{n-1},0,x^{n+1})\}$.
\end{enumerate}
This can be accomplished since $\Sigma$ meets $\p M$ orthogonally.
In these coordinates, the involution $\iota$
now takes the form $(x^1,\ldots,x^n,x^{n+1})\mapsto (x^1,\ldots,-x^n,x^{n+1})$.
Notice that, because the second fundamental form of $\p M$ vanishes, 
$\phi^{-1}(\p M\cap B_{r}(x_0))$
is contained in the hyperplane $\{(x^1,\ldots,x^{n+1})\colon x^n=0\}$.

For a radius $r^\prime<r$, we consider the $n$-dimensional ball
\begin{equation*}
B_{r^\prime}^n(0)=\{x\in T_{x_0}M\colon x^{n+1}=0,||x||< r^\prime\},
\end{equation*}
the $n$-dimensional half-ball 
$B_{r^\prime,+}^n(0)=\{x\in B_{r^\prime}^n(0)\colon x^n\geq0\},$
and the cylinder
\begin{equation*}
C_{r^\prime}(0)=\{x\in T_{x_0}M\colon (x^1,\ldots,x^n,0)\in B_{r^\prime}^n(0)\}.
\end{equation*}
For small enough $r^\prime$, we may write $\hat \Sigma\cap C_{r^\prime}(0)$ as the graph
of a function
\[
u:B_{r^\prime,+}^n(0)\to\mathbb{R},\quad \mathrm{graph}(u)=\hat\Sigma\cap C_{r^\prime}(0)
\]
where $\mathrm{graph}(u)=\{(x^1,\ldots,x^n,u(x^1,\ldots,x^n))\colon (x^1,\ldots,x^n,0)\in B_{r^\prime}^n(0)\}$.
Now we may form the doubling of $u$ to a function 
$u_\mathcal{D}:B_{r^\prime}^n(0)\to\mathbb{R}$, setting
\[
u_\mathcal{D}(x^1,\ldots,x^n)=\begin{cases}u(x^1,\ldots,x^n)&\text{ if }x^n\geq0\\
u(x^1,\ldots,x^{n-1},-x^n)&\text{ if }x^n<0.
\end{cases}
\]
To show
$\Sigma_\mathcal{D}$ is smooth at $x_0$, it suffices
to show that $u_\mathcal{D}$ is smooth along $\{x\in B_{r^\prime}^n(0)\colon x^n=0\}$.

From the free boundary condition, we have 
$\frac{\p u}{\p{x^n}}\equiv0$ on $\{x^n=0\}$ and so $u_\mathcal{D}$
has a continuous derivative on all of $B_{r^\prime}^n(0)$.
Since $\hat\Sigma$ is smooth and minimal, 
$u_\mathcal{D}$ is smooth and solves the 
minimal graph equation (\ref{eq:mean}) with respect to the metric $\hat g_\mathcal{D}$
in the strong sense on $\{x\in B_{r^\prime}^n(0)\colon x^n\neq0\}$. Moreover,
it follows from $\frac{\p u}{\p{x^n}}\equiv0$ on $\{x^n=0\}$
and the $\iota$-invariance of $\bar g_\mathcal{D}$ that
$u_\mathcal{D}$ solves the minimal graph equation weakly on the entire ball
$B^n_{r'}(0)$.

From this point, the smoothness of $u_\mathcal{D}$ is a standard application 
of tools from nonlinear elliptic PDE theory, so we will be brief 
(see \cite[Lemma 7.2]{CM}). 
Standard estimates for minimizers implies $u_\mathcal{D}\in H^2(B^n_{r^\prime}(0))$ 
(see \cite[Section 8.3.1]{Ev}).
Writing the equation (\ref{eq:mean}) in divergence form, we have 
\begin{equation}\label{eq:divmean}
\frac{\p}{\p x^i}\left(a^{ij}\frac{\p u_\mathcal{D}}{\p x^j}+b^i
u_\mathcal{D}\right)=0
\end{equation}
where the coefficients $a^{ij}$ and $b^i$ depend on $u_\mathcal{D}$ and are
only differentiable. Since $u_\mathcal{D}$ weakly solves
equation (\ref{eq:divmean}),
\begin{equation*}
\int_{B^n_{r^\prime}(0)}\left(a^{ij}\frac{\p u_\mathcal{D}}{\p x^j}+b^i u_\mathcal{D}\right)
\frac{\p\psi}{\p x^i} dx=0,
\end{equation*}
for any test function $\psi\in C_0^\infty(B^n_{r^\prime}(0))$
Taking $\psi$ to be of the form $-\frac{\p w}{\p x^k}$ for some function $w$
and integrating by parts, one finds $\frac{\p u_\mathcal{D}}{\p
x^k}$ is a weak solution of a uniformly elliptic linear equation with
$L^\infty$ coefficients for each $k=1,\ldots,n$.

Now we may apply the DeGiorgi-Nash theorem (see \cite[Theorem
8.24]{GT}) to conclude that, for each $r^{\prime\prime}<r^\prime$ there is an
$\alpha\in(0,1)$ such that $\frac{\p u_\mathcal{D}}{\p x^k}\in
C^{0,\alpha}(B^n_{r^{\prime\prime}}(0))$ for each $k=1,\ldots,n$. Now $u^D\in
C^{1,\alpha}(B^n_{r^{\prime\prime}}(0))$ and the functions $\frac{\p
u_\mathcal{D}}{\p x^k}$ solve a uniformly elliptic linear equation
with H\"older coefficients.  The Schauder estimates from Theorem
\ref{thm:schauder} allow us to conclude that $\frac{\p u_\mathcal{D}}{\p
x^k}\in C^{2,\alpha}(B_{r^\prime}(0))$.  This argument may be iterated,
see \cite[Section 8]{GT}, to conclude $u_\mathcal{D}\in
C^{k,\alpha}(B^n_{r^{\prime\prime}}(0))$ for any $k$. This finishes the proof that
$u_\mathcal{D}$ is a smooth solution to the mean curvature equation
across the doubling locus $\{x^n=0\}$ and hence $\Sigma_\mathcal{D}$
is a smooth minimal hypersurface.

The last step is to show that $\Sigma_\mathcal{D}$ is stable. Let 
$\phi\in C^\infty(\Sigma_\mathcal{D})$ define a normal variation
and write $\phi=\phi_0+\phi_1$ where $\phi_0$ is invariant under the 
involution and $\phi_1$ is anti-invariant under the involution. Now we 
will consider the second variation of the volume
of $\Sigma_\mathcal{D}$ with respect to $\phi$.
\begin{align}\label{eq:2var}
  \delta_\phi^2(\Sigma_\mathcal{D})&=
  \int_{\Sigma_\mathcal{D}}|\nabla\phi|^2-\phi^2(\Ric(\nu,\nu)+|A|^2)d\mu\notag
  \\ {}&=\int_{\Sigma_\mathcal{D}}|\nabla\phi_0|^2+2g(\nabla\phi_0,\nabla\phi_1)+|\nabla\phi_1|^2-
(\phi_0^2+2\phi_0\phi_1+\phi_1^2)(\Ric(\nu,\nu)+|A|^2)d\mu\notag\\ {}&=\delta_{\phi_0}^2(\Sigma_\mathcal{D})+\delta_{\phi_1}^2(\Sigma_\mathcal{D})+
  \int_{\Sigma_\mathcal{D}}2g(\nabla\phi_0,\nabla\phi_1)-2\phi_0\phi_1(\Ric(\nu,\nu)+|A|^2)d\mu\notag\\ {}&=2\delta_{\phi_0|_\Sigma}^2(\Sigma)+2\delta_{\phi_1|_\Sigma}^2(\Sigma)
\geq0 \notag
\end{align}
where the last equality follows from the fact that
$g(\nabla\phi_0,\nabla\phi_1)$ and $\phi_0\phi_1$ are
anti-invariant under the involution. This completes the proof of
Lemma \ref{lem:double}. 
\end{proof}
\subsection{Second fundamental form bounds}\label{section:SFF}
In this section, we will prove Step 2 in Section \ref{section:claim3}.
Let $(M_i,\bar g_i)$ and $W_i$ be as in Main Lemma.
The uniform second fundamental form bounds for the stable minimal hypersurfaces
$W_i\subset M_i$ can be reduced to a classical estimate due to Schoen-Simon
\cite{SS} for stable minimal hypersurfaces in Riemannian manifolds. 
In the following, $(M,\bar g)$ is a complete $(n+1)$-dimensional Riemannian manifold,
$x_0\in M$, $\rho_0\in(0,\mathrm{inj}_{\bar g}(x_0))$, and $\mu_1$ is a constant
satisfying 
\begin{equation}\label{eq:boundedgeometry}
\begin{array}{c}
\sup_{B_{\rho}(0)}\left|\frac{\partial \bar g_{ij}}{\partial x^k}\right|\leq\mu_1,\quad 
\sup_{B_{\rho}(0)}\left|\frac{\partial^2 \bar g_{ij}}{\partial x^k\partial x^l}\right|\leq\mu_1^2,
\end{array}
\end{equation}
on the metric ball $B_{\rho_0}(x_0)$ 
in geodesic normal coordinates
$(x^1,\ldots,x^{n+1})$ centered at $x_0$.
\begin{theorem}[Corollary 1 \cite{SS}]\label{theorem:SS}
Suppose $\Sigma$ is an oriented embedded $C^2$-hypersurface in an
$(n+1)$-dimensional Riemannian manifold $(M,\bar g)$ with
$x_0\in\overline{\Sigma}$, $\mu_1$ satisfies {\rm
(\ref{eq:boundedgeometry})}, and $\mu$ satisfies the bound
$\rho_0^{-n}\mathcal{H}^n(\Sigma\cap B_{\rho_0}(x_0))\leq\mu$.  Assume
that $\mathcal{H}^n(\Sigma\cap B_{\rho_0}(x_0))<\infty$ and
$\mathcal{H}^{n-2}(\mathrm{sing}(\Sigma)\cap B_{\rho_0}(x_0))=0$.  If
$n\leq6$ and $\Sigma$ is stable in $B_{\rho_0}(x_0)$, then
\[
\sup_{B_{\rho_0}(x_0)}|A^\Sigma|\leq\frac{C}{\rho_0},
\]
where $C$ depends only on $n$, $\mu$, and $\mu_1\rho_0$.
\end{theorem}
\begin{proof}[Proof of Step 2]
By Lemma \ref{lem:double}, the doubling $(W_i)_\mathcal{D}$ is a
smooth stable minimal hypersurface of $(M_i)_\mathcal{D}$. In
particular, the singular set of $(W_i)_\mathcal{D}$ is
empty. Moreover, the manifolds $(M_i)_\mathcal{D}$ have uniformly
bounded geometry so that the injectivity radius is uniformly bounded
from below by some $\rho_0>0$, and there is a constant $\mu_1$ so that
the bounds (\ref{eq:boundedgeometry}) hold in normal coordinates about
any $x\in (M_i)_\mathcal{D}$, any $\rho\in(0,\rho_0)$, and all
$i=1,2.\ldots$. According to Step \ref{step:volume}, there is a
constant $\mu$ such that
\[
\rho_0^{-n}\vol(W_i\cap B_\rho(x))\leq \mu
\]
for all $i=1,2,\ldots$.
Hence, we may uniformly apply Theorem \ref{theorem:SS} on any ball 
$B_{\rho_0}(x_0)\subset (M_i)_\mathcal{D}$ intersecting $W_i$
to obtain the bound in Step 2.
\end{proof}
\subsection{Generic regularity in dimension 8}\label{sec:dim8}
It is well known that codimension one volume minimizing currents, in
general, have singularities if the ambient space is of dimension 8 or
larger. However, in \cite{Smale93} N. Smale developed a method for
removing these singularities in $8$-dimensional Riemannian manifolds
by making arbitrarily small conformal changes.  In this section, we
will describe the modifications necessary to adapt his method to the
case of Theorem \ref{thm5} with $n=7$.

First, we will describe the perturbation result we will use. 
Let $M$ be a compact $(n+1)$-dimensional manifold. For $k=3,4,\ldots$, let 
$\mathcal{M}^k_0$ denote the class of $C^k$ metrics on $M$ which split isometrically as a 
product on some neighborhood of $\p M$. 
Fix a relative homology class $\alpha\in H_n(M,\p M;\mathbb{Z})$.
We will show the following.
\begin{theorem}\label{thm:dim8}
Let $g_0\in\mathcal{M}^k_0$ and $n=7$. For $\epsilon>0$, there exists a metric 
$g\in\mathcal{M}^k_0$ and a $g_0$-volume minimizing current $T$ representing $\alpha$
such that $||g-g_0||_{C^k}<\epsilon$ and $\mathrm{spt}(T)$ is smooth.
\end{theorem}
The proof of Theorem \ref{thm:dim8} follows by showing the
constructions in \cite{Smale93} can be performed on the doubled
manifold $M_{\mathcal{D}}$ (see Appendix \ref{sec:doubling}) in an
involution-invariant manner.  We proceed in two lemmas. The first
lemma holds in any dimension.
\begin{lemma}\label{lem:unique}
Let $g_0\in\mathcal{M}^k_0$ and suppose $T$ is a homologicly $g_0$-volume minimizing 
current representing $\alpha$. For $\epsilon>0$, there
is a metric $g\in\mathcal{M}^k_0$ such that $||g-g_0||_{C^k}<\epsilon$
and $T$ is the only $g$-volume minimizing current representative of $\alpha$.
\end{lemma}

\begin{proof}
Let $A$, $d\mu=\theta d\mathcal{H}^n$, and $\xi$ be the underlying 
rectifiable set, measure, and choice of 
orientation for the approximate tangent space of $A$
associated to the current $T$ (see Section \ref{sec:thm4}). 
We may write $A=\cup_{j=1}^N A_j$
where each $A_j$ are connected. Choose $p_j\in\mathrm{reg}(A_j)\setminus\p M$ and 
$\rho>0$ so that 
\[
\left(B_\rho(p_j)\cap A_j\right)\subset\left(\mathrm{reg}(A)\setminus\p M\right),
\quad j=1,\ldots,N.
\]
Perhaps restricting to smaller $\rho$, 
let $x=(x^1,\ldots,x^n)$ be geodesic normal coordinates for 
$B_\rho(p_j)\cap A_j$ and let $t$ be the signed distance on $B_\rho(p_j)$ from $A_j$
determined by $\xi$.
This gives Fermi coordinates $(t,x)$ on $B_\rho(p_j)$. 
Now fix a bump function $\eta:A\to[0,1]$ satisfying
\[
\eta(x)=\begin{cases}
1&\text{ for }x\in B_{\rho/2}(p_j)\cap A_j\\
0&\text{ for }x\in B_\rho(p_j)\setminus B_{3\rho/4}(p_j)
\end{cases}
\]
for each $j=1,\ldots N$. 
Also fix a smooth function $\phi:\mathbb{R}\to\mathbb{R}$ with
$\mathrm{spt}(\phi)\subset [-3/4,3/4]$,
\[
\phi(t)\geq0\text{ on }[-1,1],
\phi(0)=1,\text{ and }
\phi(r)<1\text{ if }r\neq0. 
\]
Consider the function $\phi_{\bar \epsilon}:M\to\mathbb{R}$ given by
\[
\phi_{\bar\epsilon}(y)=\begin{cases}
1-{\bar\epsilon}^{k+1}\phi(t/{\bar\epsilon})\eta(x)&\text{ if }y=(x,t)\in 
	B_\rho(p_j)\text{ for some }j\\
1&\text{ otherwise}
\end{cases}
\]
for $\bar\epsilon>0$ satisfying
$\mathrm{spt}(\phi_{\bar\epsilon})\subset\cup_{j=1}^NB_{3\rho/4}(p_j)$.
We have the perturbed metrics 
$g_{\bar\epsilon}=\phi_\epsilon^{\frac{2}{n}}g_0\in\mathcal{M}^k_0$.
It is straight-forward to show that there exists $\epsilon_1\in(0,\epsilon)$ such that, for any 
$\bar\epsilon\in(0,\epsilon_1]$, $T$ is the only $g_{\bar \epsilon}$-volume minimizing
representative of $\alpha$ (see \cite{Smale93}).
Perhaps restricting to smaller values of $\bar\epsilon,$ 
we may also arrange for $||g-g_{\bar \epsilon}||_{C^k}<\epsilon$. This 
completes the proof of Lemma \ref{lem:unique}.
\end{proof}
\begin{lemma}\label{lem:dim8}
Let $n=7$, $k\geq3$, $g_0\in\mathcal{M}^k$, and $\epsilon>0$. Suppose $T$ is the 
only $g_0$-volume minimizing representative of $\alpha$, then there exists 
$g\in\mathcal{M}^k$ such that $||g-g_0||_{C^k}<\epsilon$ and $\alpha$
may be represented (up to multiplicity) by a smooth $g$-volume minimizing hypersurface.
\end{lemma}
\begin{proof}
Following \cite{Smale93}, we construct a conformal factor which will
slide the minimizing current off itself in one direction and appeal to
a perturbation result for isolated singularities which allows us to
conclude that this new current has no singularity.  Write
$(M_\mathcal{D},g_{0,\mathcal{D}})$ for the doubling of $(M,g_0)$ (see
Section \ref{sec:doubling}) with involution $\iota:M_\mathcal{D}\to
M_\mathcal{D}$.  The current $T$ may also be doubled to obtain an
involution-invariant current $T_\mathcal{D}$ on $M_\mathcal{D}$.
Similarly to Section \ref{sec:doubling}, $T_\mathcal{D}$ is locally
$g_{0,\mathcal{D}}$-volume minimizing.  Let $A=\cup_{j=1}^NA_j$,
$d\mu=\theta d\mathcal{H}^7$, and $\xi$ be the underlying set,
measure, and orientation associated to $T$, as in the proof of Lemma
\ref{lem:unique}.

Let $\rho_0>0$ and fix a smooth function 
$\phi:\mathbb{R}\to\mathbb{R}$ satisfying
\begin{enumerate}
\item $\phi(-t)=-\phi(t),$
\item $\phi(t)\geq0\text{ for }t\geq0,$
\item $\phi(t)=t\text{ for }t\in[0,\frac{\rho_0}{4}],$
\item $\phi(t)=\frac{\rho_0}{2}\text{ for }t\in[\frac{\rho_0}{2},\frac{3\rho_0}{4}],$
\item $\phi(t)=0\text{ for }t\geq\rho_0.$
\end{enumerate}

Let $\{B_\rho(p_j)\}_{j=1}^N$ be a collection of disjoint metric balls in $\mathring{M}$
centered at regular points $p_j\in A_j$.
Choose $\rho_0>0$ small enough to ensure that, in Fermi coordinates $(t,x)$
for $A_j$ with $\xi$ pointing into the side corresponding to $t>0$, the function
$(t,x)\mapsto\phi(t)\eta(x)$ is supported in $\cup_{j=1}^NB_\rho(p_j)$. 
For a fixed $s\in(0,1)$ and a parameter $\bar \epsilon\in(0,1)$, consider the 
functions $u_{\bar \epsilon}:\Sigma\to\mathbb{R}$ given by
\[
u_{\bar \epsilon}(y)=\begin{cases}
1-\bar{\epsilon}^s\phi(t)\eta(x)&\text{ if }y=(t,x)\in\cup_{j=1}^NB_\rho(p_j)\\
1&\text{ otherwise.}
\end{cases}
\]
The conformal metrics $g_{\bar \epsilon}=u_{\bar \epsilon}^{\frac{2}{n}}g_0$
will be used to find the desired smooth representative. Since $g_{\bar\epsilon}$
splits as a product near $\p M$, we may consider 
the corresponding $\iota$-invariant metric $g_{\bar\epsilon,\mathcal{D}}$ on 
$M_\mathcal{D}$.

For sake of contradiction, suppose that there is a sequence 
$\bar \epsilon_i\to0$ and homologically $g_{\bar \epsilon_i}$-volume minimizing
currents $T_i$ representing $\alpha$ with $\mathrm{sing}(T_i)\neq\emptyset$
for all $i=1,2,\ldots.$
Since $\mathbf{M}(T_i)$ is uniformly bounded in $i$, 
$T_i$ weakly converges to some homologically $g_0$-volume
minimizing current $T_\infty$ which also represents $\alpha$. 
Since $T$ is assumed to be the unique such current, we must have $T_\infty=T$.
Write $P_i$, $d\mu_i$, and $\xi_i$ for the set, measure, and orientation
corresponding to $T_i$ for $i=1,2,\ldots$.
Let $Q_i$ be a connected component of $P_i$ with 
$\mathrm{sing}(Q_i)\neq\emptyset$ for each $i=1,2,\ldots$.
Now $Q_i$ converges in the Hausdorff sense to some sheet 
$Q$ of $T$. By the Allard regularity theorem \cite{All},
this convergence is smooth away from $\mathrm{sing}(Q)$.
Hence, after passing to a subsequence, $y_i$ converges to 
some $y\in \mathrm{sing}(Q)$.

In terms of the doubled manifold, the $\iota$-invariant currents
$T_{i,\mathcal{D}}$ are homologically
$g_{\bar\epsilon_i,\mathcal{D}}$-volume minimizing, $T_{i,\mathcal{D}}$
weakly converge to $T_{0,\mathcal{D}}$, and the doubled sets $Q_{i,\mathcal{D}}$
converge to $Q_{\mathcal{D}}$ smoothly away from $\mathrm{sing}(Q_\mathcal{D})$.
Now let $\mathcal{N}\subset M_\mathcal{D}$ be a small distance neighborhood of 
$Q_\mathcal{D}$ 
so that $\mathcal{N}\setminus Q_\mathcal{D}$ consists of two disjoint, open 
sets $\mathcal{N}_-$ and $\mathcal{N}_+$ on which the signed distance to
$Q_\mathcal{D}$ is negative and positive, respectively. 
In the doubled manifold, we may directly apply the following results from
\cite{Smale93}.
\begin{lemma}\label{lem:dim8reg}\cite[Proposition 1.6]{Smale93}
For large $i$, we have
\begin{enumerate}
\item $Q_{i,\mathcal{D}}\cap\mathcal{N}_-=\emptyset$
\item $Q_{i,\mathcal{D}}\cap\mathcal{N}_+\setminus
	\mathrm{spt}(\phi_{\epsilon_i}\eta)_\mathcal{D}\neq\emptyset$.
\end{enumerate}
\end{lemma}
In light of Lemma \ref{lem:dim8reg},
the Simon maximum principle \cite{Si87} shows  
\[
\left(Q_{i,\mathcal{D}}\setminus \mathrm{spt}(\phi_i\eta)_\mathcal{D}\right)\subset
\left(\mathcal{N}_+\setminus \mathrm{spt}(\phi_i\eta)_\mathcal{D}\right)
\]
for each $i=1,2,\ldots$.
Recalling that $Q_{i,\mathcal{D}}$ converges to $Q_\mathcal{D}$ in the Hausdorff 
distance, we may apply the perturbation result \cite[Theorem 5.6]{HS} to
conclude that $Q_{i,\mathcal{D}}$ is smooth for sufficiently large $i$. This
contradiction finishes the proof of Lemma \ref{lem:dim8}.
\end{proof}
Theorem \ref{thm:dim8} follows by first applying Lemma \ref{lem:unique} 
to approximate $g_0$ with a metric $g_1$ supporting a unique minimizing 
representative of $\alpha$ then applying Lemma \ref{lem:dim8} to approximate $g_1$
with a metric $g_2$ and obtain a $g_2$-volume minimizing representative of $\alpha$.
\begin{proof}[Proof of Theorem \ref{thm5} for $n=7$]
We will closely follow the argument presented in Section \ref{sec4}.
Let  $(Z,\bar g,\bar \gamma):
(Y_0,g_0,\gamma_0) \rightsquigarrow (Y_1,g_1,\gamma_1)$ be a psc-bordism
and let $(Z_i,\bar g_i,\bar \gamma_i)$ be the corresponding $i$-collaring 
for $i=1,2,\ldots$. As usual, we denote by $\bar \alpha_i\in H_7(Z_i,\p Z_i;\mathbb{Z})$
the Poincar\'e dual to $\bar \gamma_i$.

For each $i=1,2,\ldots$, we apply Theorem \ref{thm:dim8}
to obtain a metric $\hat g_i$ on $Z_i$ so that
\[
||\hat g_i-\bar g_i||_{C^i_{\bar g_i}}\leq\frac{1}{i}
\]
and $\bar \alpha_i$ can be represented by a smooth $\hat g_i$-volume minimizing 
hypersurface $W_i$. It follows from the proofs of
Lemmas \ref{lem:unique} and \ref{lem:dim8} that $\hat g_i$ can and will be
chosen so that $\{\hat g_i\neq \bar g_i\}\subset M_1\subset M_i$ for $i=1,2,\ldots$.
Indeed, the perturbations required to form $\hat g_i$
are supported on balls centered about chosen regular points of
$\bar g_i$-volume minimizing currents and
one can always find regular points of minimizers of $\bar \alpha_i$ in $M_1\subset M_i$.
Evidently, $\hat g_i$ has positive scalar curvature for all
sufficiently large $i$. 
Since $\hat g_i=\bar g_i$ on $Y\times[-i,0]\subset Z_i$, 
the proof of the Main Lemma shows that there is a subconvergence
\[
(Z_i,W_i,\hat g_i,\mathsf{S}_i)\rightarrow
(Y\times(-\infty,0],X\times(-\infty,0],g+dt^2,\mathsf{S})
\]
where $Y, X, g, \mathsf{S}_i$, and $\mathsf{S}_\infty$ are defined as in Section \ref{sec4}.
One can now directly apply the argument from \ref{sec:longcollars} to finish the proof 
of Theorem \ref{thm5} for $n=7$.
\end{proof}

\bibliographystyle{elsarticle-num}

\end{document}